\documentclass{article}
\usepackage[a4paper,margin=2cm]{geometry}
\usepackage{amsmath,amssymb,amsthm}

\newcommand\N{\mathbb N}
\newcommand\Z{\mathbb Z}
\newcommand\Q{\mathbb Q}
\newcommand\R{\mathbb R}
\newcommand\C{\mathbb C}
\newcommand\F{\mathbb F}
\newcommand\ro{\mathcal O}
\newcommand\h{\mathbb H}
\DeclareMathOperator\Gal{Gal}
\DeclareMathOperator\End{End}
\DeclareMathOperator\Ker{Ker}
\DeclareMathOperator\re{Re}
\DeclareMathOperator\im{Im}

\newcommand{\qsqrt}[1]{(\sqrt{q_1^*},\dots,\sqrt{q_{#1}^*})}
\newcommand{\qtsqrt}{\qsqrt{t}}
\newcommand{\tqsqrt}[1]{(\sqrt{\tilde q_1},\dots,\sqrt{\tilde q_{#1}})}
\newcommand{\tqrsqrt}{\tqsqrt{r}}

\newtheorem{theorem}{Theorem}
\newtheorem{lemma}{Lemma}
\newtheorem{statement}{Statement}
\newtheorem{corollary}{Corollary}

\title{Method for constructing elliptic curves using complex multiplication and its optimizations}
\author{Grechnikov~E.~A.}

\begin{document}
\maketitle
\begin{abstract}
Elliptic curves over finite fields with predefined conditions in the order are practically constructed
using the theory of complex multiplication. The stage with longest calculations in this
method reconstructs some polynomial with integer coefficients. We will prove theoretical
results and give a detailed account of the method itself and how one can use a divisor of the
mentioned polynomial with coefficients in some extension of the field of rational numbers.
\end{abstract}
\section{Introduction}
Elliptic curves play an important role in a variety of different applications.
For example, elliptic curves form a basis for some public-key cryptosystems
\cite{koblitz, miller}, primality tests \cite{lenstra} and factorization \cite{atkinmorain}
of rational integers. The applications use elliptic curves over finite fields with the order
satisfying several restrictions. For instance, for cryptographical applications
the order should be a prime number or, at least, should have a large prime divisor.

One of methods for constructing elliptic curves with predefined restrictions on the order
is the following. First, we generate an equation of an elliptic curve with random coefficients.
Next, we calculate the order of the generated curve and check whether the order satisfies
the predefined conditions. If so, the construction is done; otherwise, we repeat the process
from the beginning. Possible values for arising orders are distributed approximately
uniformly (the precise statement for prime fields with characteristic greater than 3
can be found in \cite{lenstra}). The calculation of order of an elliptic curve has
a polynomial complexity. However, in practice the complexity grows quite fast,
so this method is quite slow.

The complex multiplication gives another, more practical method for constructing elliptic
curves with predefined restrictions on the order. This article is devoted to the
complex multiplication method. Here we start with calculating an order satisfying
the predefined conditions and then construct an elliptic curve with this order.
Section \ref{method} describes the details and some known optimizations.

We suggest a new optimization for calculations. Further theoretical results,
used by this optimization, are
proved in Sections \ref{omegaprop} and \ref{el.integers}. We describe
the overall approach in Section \ref{divider} and the details in
next sections.

\section{The CM method}\label{method}
\subsection{Theoretical basis}
Hereafter we always assume that $D\in\Z$ satisfies the following condition:
\begin{equation}\label{el.dbase}
D<0\mbox{ and either }D\equiv0\pmod4\mbox{ or }D\equiv1\pmod4.
\end{equation}
Consider the field $K=\Q(\sqrt{D})$. Let $d$ be the discriminant of $K$.
Then $d<0$ and
\begin{itemize}
\item either
\begin{equation}\label{eq.d1}
d\equiv1\pmod4\mbox{ and }d\mbox{ is square-free},
\end{equation}
\item or
\begin{equation}\label{eq.d2}
d\equiv0\pmod4,\qquad\frac d4\mbox{ is square-free},\qquad\frac d4\not\equiv1\pmod4.
\end{equation}
\end{itemize}
In addition, $D=f^2d$, where $f\in\N$. Let $\ro=\Z\left[\frac{d+\sqrt d}2\right]$ be
the ring of algebraic integers in the field $K$. Let $\ro_D=\Z\left[\frac{D+\sqrt D}2\right]$
be the order in $\ro$ with conductor $f$. For any number field $M$ we denote the
ring of integers for $M$ by $\ro_M$; e.g. $\ro_K=\ro$.

We define a \textit{fractional ideal} of the order $\ro_D$ as a subset of $K$ which is
a finitely generated $\ro_D$-module and contains a nonzero element. We define
an \textit{ideal} of $\ro_D$ as a fractional ideal which is a subset of $\ro_D$ (this
is the same as the standard definition of ring ideal except for $\{0\}$,
which is not considered).
We define a \textit{proper} (fractional) ideal of the order $\ro_D$ as a (fractional)
ideal $\mathfrak a$ such that
$\{\beta\in K:\beta\mathfrak a\subset\mathfrak a\}=\ro_D$. All proper fractional
ideals of an order form an abelian group under the multiplication of ideals
(\cite[\S7]{cox}). We denote this group by $I(\ro_D)$.
It is easy to see that \textit{principal} fractional ideals, i.e. $\alpha\ro_D$
with $\alpha\in K^*$, form a subgroup in $I(\ro_D)$. We denote this subgroup
by $P(\ro_D)$. Two ideals $\mathfrak a$ and $\mathfrak b$ are \textit{equivalent}
when they differ by multiplication by a principal ideal. It is easy to see
that this relation is an equivalence relation; we denote it by
$\mathfrak a\sim\mathfrak b$. For the sake of brevity we call a class of
equivalent proper fractional ideals an \textit{ideal class}.
Since $P(\ro_D)$ is a subgroup, all ideal classes form a factorgroup
$\mathcal H_D=I(\ro_D)/P(\ro_D)$. It is called the \textit{ideal class group}
of $\ro_D$. The ideal class group is a finite abelian group (\cite[\S7]{cox}).
Since $\ro_D$ and $K$ are invariant under complex conjugation, the
conjugation of a fractional ideal as a set is itself a fractional ideal;
the complex conjugation induces a well-defined operation on $\mathcal H_D$.

We define a \textit{quadratic form} as an expression of the form $Ax^2+Bxy+Cy^2$,
where $A,B,C\in\Z$. We also use $(A,B,C)$ as another notation for the same
quadratic form. We define a \textit{discriminant} of the quadratic form as
$B^2-4AC$. Two forms are \textit{equivalent} if one can be transformed to another
using change of variables $x'=ax+by,y'=cx+dy$ with $a,b,c,d\in\Z$, $ad-bc=1$;
it is easy to see that this is indeed an equivalence relation. The quadratic
form $(A,B,C)$ is \textit{positive definite} if $A>0$ and $B^2-4AC<0$;
it is \textit{primitive} if $\gcd(A,B,C)=1$. Hereafter we consider only
primitive positive definite forms of the discriminant $D$, calling them
just \textit{forms} for the sake of brevity. We define the \textit{root}
of the form as the (only) root $\tau$ of the equation $A\tau^2+B\tau+C=0$
from the upper half-plane $\h=\{z\in\C:\im z>0\}$, i.e.
$\tau=\frac{-B+\sqrt{D}}{2A}\in K\cap\h$. The form $(A,B,C)$ is
\textit{reduced} when $|B|\le A\le C$ and if $B<0$, then $|B|<A<C$.
Every form is equivalent to exactly one reduced form (\cite[Theorem 2.8]{cox}).

There is one-to-one correspondence between elements of the group $\mathcal H_D$
and reduced forms. We denote this correspondence by $\mathfrak h$. Namely
(\cite[Theorem 7.7]{cox}), a form $\xi=(A,B,C)$ with the root $\tau$
corresponds to a class $\mathfrak h(\xi)=\mathfrak h(A,B,C)$ of $\ro_D$-ideals
containing $\langle1,\tau\rangle_\Z$ (which is a proper fractional $\ro_D$-ideal),
and two equivalent forms correspond to the same ideal class.

It is easy to enumerate all reduced forms: obviously, such a form has
$|B|\le A\le\sqrt{\frac{|D|}3}$ and for fixed $A,B$ there exists at most one $C$.
So reduced forms give a convenient way to organize elements of $\mathcal H_D$.

The classical $j$-invariant is the function from the upper half-plane $\h$ to $\C$
(\cite[\S46]{weber}). It can be also defined on lattices in $\C$ (\cite[\S10]{cox})
so that it does not change when a lattice is multiplied by any nonzero complex number
and $j(\tau)=j(\langle1,\tau\rangle_\Z)$ for $\tau\in\h$. Any proper fractional
$\ro_D$-ideal $\mathfrak a$ is also a lattice in $\C$. Obviously, $j(\mathfrak a)$
depends only on the ideal class of $\mathfrak a$. From the computational point of view,
if the fractional ideal $\mathfrak a$ belongs to the ideal class corresponding to
the form $Ax^2+Bxy+Cy^2$ with the root $\tau$ (i.e. $\mathfrak a\sim\langle1,\tau\rangle_\Z$),
then $j(\mathfrak a)=j(\tau)$.

For any elliptic curve and $n\in\Z$ we define the map $[n]$ which maps a point $P$ to $nP$.
In particular, $[1]$ is the identity map. We define a \textit{isogeny} of two elliptic curves
as a morphism (in the sense of algebraic geometry) which maps the infinite point
of the first curve to the infinite point of the second curve. We define an \textit{endomorphism}
of an elliptic curve as an isogeny of the curve to itself. For any $n\in\Z$ and any
elliptic curve the map $[n]$ is an endomorphism (\cite[Example III.4.1]{silverman})
and commutes with any other endomorphism (because any isogeny is a homomorphism of groups of points
due to \cite[Example III.4.8]{silverman}); the ring of endomorphisms of any elliptic curve
is a $\Z$-module with an action $n\varphi = [n]\circ\varphi$, where $n\in\Z$, $\varphi$ is
an endomorphism. Endomorphisms $\{[n]:n\in\Z\}$ form the ring isomorphic to $\Z$
(\cite[Proposition III.4.2]{silverman}).

The ring of endomorphisms of an elliptic curve over $\C$ is either equal to
$\{[n]:n\in\Z\}$ or isomorphic to an order in some imaginary quadratic field
(\cite[Corollary III.9.4 and Exercise 3.18b]{silverman}). In the last case
the curve is said to have \textit{complex multiplication} by this order.
There exists exactly $|\mathcal H_D|$ nonisomorphic elliptic curves with
complex multiplication by $\ro_D$ (\cite[Proposition C.11.1]{silverman}).
These curves can be characterized as follows: the $j$-invariant of a curve
equals one of values of modular $j$-invariant in an ideal representing
an ideal class for $\ro_D$. These values are called \textit{singular values}
(of the function $j$).
Any singular value generates over $K$ the same field $L=L(D)=K(j(\mathfrak a))$,
which is called the \textit{ring class field} for $\ro_D$ (\cite[Theorem 11.1]{cox}).
The Galois group of the extension $L/K$ is isomorphic to $\mathcal H_D$
(\cite[\S9]{cox}). We denote the canonical isomorphism by $\Omega$;
$\Omega$ maps an ideal class containing $\mathfrak b$ to the automorphism
mapping $j(\mathfrak a)$ to $j(\mathfrak a\mathfrak b^{-1})$
(\cite[Corollary 11.37]{cox}). The complex conjugation acts as follows:
$\overline{j(\mathfrak a)}=j(\overline{\mathfrak a})$ by the definition of $j$
(\cite[\S10]{cox}), $\mathfrak a\overline{\mathfrak a}\sim\ro_D$
(\cite[(7.6)]{cox}), therefore, $\overline{j(\mathfrak a)}=j(\mathfrak a^{-1})$.

Let us consider the polynomial $H_D[j](x)=\prod_{i=1}^h(x-j(\alpha_i))$, where
$h=|\mathcal H_D|$ and $\alpha_i$ represent all ideal classes of $\ro_D$.
The coefficients of $H_D$ are elements of $L$, are invariant under
the action of $\Gal(L/K)$ and complex conjugation, therefore, they lie in $\Q$.
Moreover, the values $j(\alpha_i)$ are algebraic integers (\cite[Theorem 11.1]{cox}),
so $H_D[j](x)\in\Z[x]$.

Let $p$ be a prime number, $n$ a natural number, $q=p^n$. Let $E$ be
an elliptic curve defined over the finite field $\F_q$. Unless explicitly
specified, we consider $\overline\F_q$-points and $\overline\F_q$-endomorphisms
of the curve $E$. The order of the curve is the number of $\F_q$-points.
The ring of endomorphisms $\End(E)$ is isomorphic either to an order in
some imaginary quadratic field or to an order in some quaternion algebra over $\Q$
(\cite[Corollary III.9.4 and Theorem V.3.1]{silverman}). In the last case the
curve $E$ is called \textit{supersingular}, and we are not interested in these.
In the first case $\End(E)\cong\ro_D$ for some $D$ satisfying \eqref{el.dbase};
$\End(E)=\langle[1],\alpha\rangle_\Z$, where $\alpha$ is some endomorphism of $E$.
It appears (\cite[Theorem 13.14]{lang}) that the curve and one endomorphism $\alpha$
can be "lifted" to $\C$ in the following sense: there exists a number field $L'$,
an elliptic curve $E'$ defined over $L'$, an endomorphism $\alpha'$ of $E'$,
an ideal $\mathfrak B'\subset\ro_{L'}$ lying above $p$ (i.e. $\mathfrak B'\cap\Z=p\Z$)
and a reduction of $E'$ modulo $\mathfrak B'$ so that the reduced curve is
isomorphic to $E$ and $\alpha'$ corresponds to $\alpha$ under the reduction.
Since $\alpha\not\in\{[n]:n\in\Z\}$, we have $\alpha'\not\in\{[n]:n\in\Z\}$,
so $\End(E')\not\cong\Z$ and $E'$ has a complex multiplication by some order
in some imaginary quadratic field. Due to properties of reduction (\cite[Theorem 13.12]{lang}),
it induces an isomorphism of $\End(E')$ to a subring in $\End(E)$. Since
$\alpha'$ reduces to $\alpha$, we have $\End(E')\cong\End(E)\cong\ro_D$.

The ring $\End(E)$ contains a Frobenius isogeny $Fr:(x,y)\mapsto(x^q,y^q)$ and
the dual isogeny $\widehat{Fr}$. We have $Fr\circ\widehat{Fr}=[q]$
(\cite[Theorem III.6.2 and Proposition 2.11]{silverman}) and
$[|E(\F_q)|]=[|\Ker([1]-Fr)|]=([1]-Fr)\circ([1]-\widehat{Fr})$
(the first equality follows from the fact that $Fr$ fixes $\F_q$-points
and only them; the second equality follows from
\cite[Theorem III.4.10, Corollary III.5.5, Theorem III.6.2]{silverman}).
Let $\pi\in\ro_D\cong\End(E)$ be the element corresponding to $Fr$
and $\overline\pi$ be the element corresponding to $\widehat{Fr}$.
Then $\pi\overline\pi=q$ and $(1-\pi)(1-\overline\pi)=|E(\F_q)|$.
In particular, if $\pi\not\in\R$, these equations imply that $\overline\pi$
is indeed a complex conjugate to $\pi$; otherwise, $\pi\in\R\cap\ro_D=\Z$,
so $Fr=[\pi]$ and in this case $\widehat{Fr}=Fr$ (\cite[Theorem III.6.2]{silverman}),
hence $\overline\pi$ is equal to complex conjugate to $\pi$ too.

Since $\pi\in\ro_D$, there exist $u,v\in\Z$ such that $\pi=\frac{u+v\sqrt{D}}2$.
Then $\overline\pi=\frac{u-v\sqrt{D}}2$ and $q=\pi\overline\pi=\frac{u^2-Dv^2}4$ or
$$
4q=u^2+|D|v^2.
$$
The order of $E$ is $1-\pi-\overline\pi+\pi\overline\pi=q+1-u$.
Due to \cite[Exercise 5.10]{silverman} the non-supersingularity of $E$ implies $\gcd(q,u)=1$.

Let us sum up the above. If $E$ is a non-supersingular elliptic curve over $\F_q$,
then there exist an integer $D$, a number field $L'$ and an elliptic curve $E'$ over $L'$
such that
\begin{itemize}
\item $E'$ has a complex multiplication by $\ro_D$,
\item there exists a reduction of $E'$ isomorphic to $E$,
\item the order of $E$ is $q+1-u$, where $u\in\Z$ is such that for some $v\in\Z$
the equality $4q=u^2+|D|v^2$ holds.
\end{itemize}

\subsection{Basic algorithm}\label{basemethod}
We want to go in the other direction and construct such curves $E$. In order to do this,
we implement the following scheme.
\begin{enumerate}
\item Select the numbers $q=p^n$, $p$ is a prime, and $\hat u,\hat v,D\in\Z$ such that
$D$ satisfies \eqref{el.dbase},
\begin{equation}\label{el.order}
4q=\hat u^2+|D|\hat v^2,
\end{equation}
\begin{equation}\label{el.uprime}
\gcd(\hat u,q)=1,
\end{equation}
and the field size $q$ and the order of a future curve $q+1-\hat u$ satisfy
the predefined restrictions required by concrete applications.
\item Calculate the polynomial $H_D[j](x)$.
\item Reduce the polynomial $H_D[j](x)$ modulo $p$. Obtain the polynomial over
$\F_p\subset\F_q$; this polynomial (as we will show) splits into linear factors
in $\F_q$. Calculate any root of this polynomial. Generate an
elliptic curve $E''$ over $\F_q$ such that its $j$-invariant equals the found root.
\item The curve $E''$ has complex multiplication by $\ro_D$. An isomorphism
does not change the ring of complex multiplication, but can change the number
of $\F_q$-points. Construct the curve isomorphic to $E''$ with the order $q+1-\hat u$.
\end{enumerate}
We define the \textit{Kronecker symbol} $\left(\frac ab\right)$, where $a\in\Z$,
$b\in\N$, as follows. If $b$ is an odd prime, the Kronecker symbol equals
the Legendre symbol. If $b=2$, the Kronecker symbol is defined only for
$a\equiv0,1\pmod4$ and equals
$$
\left(\frac a2\right)=\begin{cases}
1,&\mbox{if }a\equiv1\pmod 8,\\
-1,&\mbox{if }a\equiv5\pmod 8,\\
0,&\mbox{if }a\equiv0\pmod4.
\end{cases}
$$
\label{el.kronekerdef}
In the general case the Kronecker symbol is defined as being multiplicative in $b$.

The conditions \eqref{el.order} and \eqref{el.uprime} impose quite strong restrictions
on $q$ and $p$. In particular, the following lemma holds.
\begin{lemma}\label{el.stupid}
Let $d<0$ satisfy one of conditions \eqref{eq.d1} or \eqref{eq.d2},
$D=df^2$. Assume that $q=p^n$ and some integers $u$, $v$ satisfy the conditions
$$4q=u^2+|D|v^2,\qquad \gcd(q,u)=1.$$
Then
\begin{enumerate}
\item
\begin{equation}\label{el.dsquare}
\left(\frac Dp\right)=\left(\frac dp\right)=1;
\end{equation}
\item $p\nmid f$;
\item $p\ro=\mathfrak p\overline{\mathfrak p}$, where $\mathfrak p\ne\overline{\mathfrak p}$ are prime ideals of $\ro$;
\item $\frac{u+v\sqrt{D}}2\in\ro_D$;
\item $$\frac{u+v\sqrt{D}}2\ro=\mathfrak p^n\mbox{ or }\frac{u+v\sqrt{D}}2\ro=\overline{\mathfrak p}^n.$$
\end{enumerate}
\end{lemma}
\begin{proof}
The equality $4q=u^2+|D|v^2$ and the condition $p\nmid u$ imply that $p\nmid D$ and $p\nmid v$.
Reducing modulo $p$, we obtain $u^2-Dv^2\equiv0\pmod p$, so $D\equiv(uv^{-1})^2\pmod p$.
To conclude the proof of the first assertion, it remains to note that $D=df^2$.

The second assertion follows obviously from $p\nmid D$.

The third assertion follows from the first one due to the well-known fact from the theory
of quadratic fields (e.g. \cite[Propositions 13.1.3 and 13.1.4]{ireland}).

To prove the fourth assertion, we reduce the equality $4q=u^2+|D|v^2$ modulo 2.
If $D$ is even, then $u$ is even, $\ro_D=\Z\left[\frac{D+\sqrt{D}}2\right]=\Z\left[\frac{\sqrt{D}}2\right]$
and $\frac{u+v\sqrt{D}}2=\frac u2+v\frac{\sqrt D}2\in\ro_D$. If $D$ is odd, then $u\equiv v\pmod2$,
$\ro_D=\Z\left[\frac{D+\sqrt D}2\right]=\Z\left[\frac{1+\sqrt D}2\right]$ and
$\frac{u+v\sqrt D}2=\frac{u-v}2+v\frac{1+\sqrt D}2\in\ro_D$.

Finally, note that
$$
\frac{u+v\sqrt{D}}2\ro\cdot\frac{u-v\sqrt D}2\ro=q\ro=p^n\ro=\mathfrak p^n\overline{\mathfrak p}^n.
$$
Since $p\nmid u$, we have $\mathfrak p\overline{\mathfrak p}=p\ro\nmid\frac{u\pm v\sqrt{D}}2\ro$.
The last assertion follows from the uniqueness of the factorization to prime ideals in $\ro$.
\end{proof}

Now we discuss some implementation details of the generic scheme.

The implementation of the first stage depends on restrictions for the field size $q$
and the curve order.

If $q$ is fixed, scan over integers $D$ satisfying \eqref{el.dbase}. For every $D$,
first check the necessary condition \eqref{el.dsquare}; if it does not hold, continue to
the next $D$. Assume that $D$ satisfies \eqref{el.dsquare}. Apply the Cornacchia algorithm
(\cite{cornacchia}) that solves the equation $x^2+|D|y^2=m$, to $m=4q$. If there is no
solution, continue to the next $D$. If a solution is found, check whether $q+1\pm x$
satisfies the restrictions for the order.

If $q$ is not fixed, it is more efficiently to fix $D$ instead of the previous method.
First, fix $D$ satisfying \eqref{el.dbase}. Next, generate $\hat u$, $\hat v$ at
random and calculate $q$ from \eqref{el.order} and $q+1\pm\hat u$; repeat until
the required restrictions are met. Some improvements of this method are
suggested in \cite{baier} and \cite{efficientgen}. In essence, these improvements
implement the following idea: one can select parameters $\hat u$, $\hat v$ less
randomly and guarantee the absence of small prime divisors of
$q=\frac{\hat u^2+|D|\hat v^2}4$ and $q+1\pm\hat u$ (or, at least, decrease the
probability of such divisors). As an example, assume the following restrictions:
$p$ is odd and one of $q+1\pm u$ is an odd prime
(that is the case in \cite{efficientgen}). It is easy to see that $D\equiv5\pmod 8$
and $\hat u$, $\hat v$ must be odd. \cite{efficientgen} suggests starting from
$\hat u=210\hat u_0+1$, $\hat v=210\hat v_0+105$, $\hat u_0$, $\hat v_0$ are random
integers; if the initial values are bad, continue with adding to $\hat u$ numbers
106 and $104 = 210 - 106$ in turn. Note that $210=105\cdot2=2\cdot3\cdot5\cdot7$.
This choice guarantees that $\frac{\hat u^2+|D|\hat v^2}4$ and one of $q+1\pm\hat u$
do not divide by 2, 3, 5, 7. The method from \cite{baier} uses more small divisors
and is cumbersome, so we do not quote it here. The performance of different
methods is compared in \cite{efficientgen}.

The second stage consists of calculating the polynomial $H_D[j](x)$.
Enumerate all reduced forms (there are $h=|\mathcal H_D|$ of them).
Calculate their roots $\tau_1,\dots,\tau_h$. Calculate the values
$j(\tau_1),\dots,j(\tau_h)$ as complex numbers with sufficiently large precision.
Calculate the coefficients of the polynom $H_D[j](x)$ approximately.
If the precision is large enough, then possible error in coefficients
is less than $\frac12$ and the exact values (which are integer numbers) can be calculated
by rounding.

For a number field $M$, a prime ideal $\mathfrak C\subset\ro_M$ and $z\in\ro_M$
we denote by $R_{\mathfrak C}(z)$ the reduction of $z$ modulo $\mathfrak C$.
So $R_{\mathfrak C}$ is a map from $\ro_M$ to a finite field. The map $R_{\mathfrak C}$
also acts on polynomials from $\ro_M[x]$, reducing every coefficient.

For the third stage we must show that the polynomial
$$R_{p\Z}(H_D[j](x))$$
splits into linear factors in $\F_q$. Also we must construct an elliptic
curve by its $j$-invariant.

Lemma \ref{el.stupid} implies that $p\ro=\mathfrak p\overline{\mathfrak p}$.
Let $\mathfrak B\subset\ro_L$ be a prime ideal lying above $\mathfrak p$. Since
$\mathfrak B\cap\Z=\mathfrak p\cap\Z=p\Z$, we have
\begin{equation}\label{el.hdreduction}
R_{p\Z}(H_D[j](x))=R_{\mathfrak B}(H_D[j](x))=\prod_{i=1}^h(x-R_{\mathfrak B}(j(\alpha_i)))
\end{equation}
so $R_{p\Z}(H_D[j](x))$ splits into linear factors in $\ro_L/\mathfrak B$. Therefore,
it remains to prove the following theorem.
\begin{theorem}\label{el.subsetinq}
\begin{equation}
\ro_L/\mathfrak B\subset\F_q.
\end{equation}
\end{theorem}
\begin{proof}
Let $\mathfrak c\subset\ro$ be a prime ideal unramified in $L$.
Let $\mathfrak C\subset\ro_L$ be a prime ideal lying above $\mathfrak c$.
Let $\left(\frac{L/K}{\mathfrak C}\right)$ denote the \textit{Artin symbol}
\cite[\S5]{cox} (it is defined for any Galois extension $K\subset L$,
but we use it only for the fields $K$ and $L$ defined above),
namely, the unique (\cite[Lemma 5.19]{cox}) element $\sigma\in\Gal(L/K)$
such that $\sigma(\alpha)\equiv\alpha^{Norm(\mathfrak c)}\pmod{\mathfrak C}$
for any $\alpha\in\ro_L$. Since $\Gal(L/K)\cong\mathcal H_D$ is Abelian,
the Artin symbol depends only on $\mathfrak c$ (\cite[Corollary 5.21]{cox})
and can be denoted as $\left(\frac{L/K}{\mathfrak c}\right)$. For
a fractional $\ro$-ideal $\mathfrak b=\mathfrak c_1^{s_1}\dots\mathfrak c_k^{s_k}$
we define
$\left(\frac{L/K}{\mathfrak b}\right)=\left(\frac{L/K}{\mathfrak c_1}\right)^{s_1}
\dots\left(\frac{L/K}{\mathfrak c_k}\right)^{s_k}$. The map $\left(\frac{L/K}{\cdot}\right)$
is a homomorphism from the group of those fractional $\ro$-ideals whose factorization
does not contain prime ideals ramified in $L$, to the group $\Gal(L/K)$. This homomorphism
is called the \textit{Artin map}.

Let $P_{K,\Z}(f)$ (\cite[\S9]{cox}) denote the subgroup of fractional $\ro$-ideals generated
by principal ideals of the form $\alpha\ro$, $\alpha\in\ro$, $\alpha\equiv a\pmod{f\ro}$
for some $a\in\Z$ with $\gcd(a,f)=1$. According to \cite[\S9]{cox}, the ring class field $L$
for $\ro_D$ is the unique abelian extension of $K$ such that
\begin{itemize}
\item all prime ideals $\ro$ ramified in $L$ divide $f\ro$ (consequently, all ideals
from $P_{K,\Z}(f)$ are unramified in $L$: if $\alpha\equiv a\pmod{f\ro}$, then
$\gcd(\alpha\ro,f\ro)=\gcd(a\ro,f\ro)=\gcd(a,f)\ro=\ro$, so $\alpha\ro$ is prime to $f\ro$),
\item the kernel of the Artin map is $P_{K,\Z}(f)$.
\end{itemize}
Let $\hat\pi=\frac{\hat u+\hat v\sqrt D}2$; Lemma \ref{el.stupid} implies that either
$\mathfrak p^n=\hat\pi\ro$ or $\mathfrak p^n=\overline{\hat\pi}\ro$. In both cases
the ideal $\mathfrak p^n$ is principal and lies in $P_{K,\Z}(f)$ (this is easy to see
from $\hat\pi=\frac{\hat u+\hat v\sqrt D}2\equiv\frac{\hat u-f\hat vd}2\pmod{f\ro}$,
$\overline{\hat\pi}=\frac{\hat u-\hat v\sqrt D}2\equiv\frac{\hat u+f\hat vd}2\pmod{f\ro}$,
from the definition of $P_{K,\Z}(f)$, Lemma \ref{el.stupid} and \eqref{el.order}).
Therefore, $\mathfrak p^n$ lies in the kernel of the Artin map. Equivalently,
$\left(\frac{L/K}{\mathfrak p}\right)^n=Id$. The automorphism $\left(\frac{L/K}{\mathfrak p}\right)$
acts on $\ro_L/\mathfrak B$ as $x\mapsto x^{Norm(\mathfrak p)}=x^p$, so
its $n$-th power acts as $x\mapsto x^{p^n}=x^q$. This means that the operation
$x\mapsto x^q$ acts trivially on $\ro_L/\mathfrak B$, which is possible only if $\ro_L/\mathfrak B\subset\F_q$.
\end{proof}

Using formulas from \cite[Proposition A.1.1]{silverman}, it is easy to check
that for $j\in\F_q$ or $j\in\C$ the following curves, defined over $\F_q$ or $\C$
respectively, have the $j$-invariant equal to $j$:
\begin{itemize}
\item if the field characteristic is 0 or greater than 3, when $j\ne0$, $j\ne1728$:
$y^2=x^3+3cx+2c$, where $c=\frac j{1728-j}$;
\item if the field characteristic is 0 or greater than 3, when $j=0$: $y^2=x^3+1$;
\item if the field characteristic is 0 or greater than 3, when $j=1728$: $y^2=x^3+x$;
\item over the field $\F_q$ of the characteristic 2, when $j\in\F_q^*$: $y^2+xy=x^3+j^{-1}$;
\item over the field $\F_q$ of the characteristic 3, when $j\in\F_q^*$: $y^2=x^3+x^2-j^{-1}$.
\end{itemize}
The missing cases with $j=0$ in characteristics 2 and 3 correspond to supersingular curves
(\cite[Exercise 5.7, Theorem 4.1]{silverman}), so they cannot arise.

For the fourth stage we must show that the curve $E''$ (which is constructed in the third stage)
has complex multiplication by $\ro_D$ (in particular, it is non-supersingular).

It follows from the construction of the curve $E''$ and from \eqref{el.hdreduction}
that the $j$-invariant of $E''$ equals $R_{\mathfrak B}(j(\mathfrak a))$,
where $\mathfrak a$ is some proper fractional $\ro_D$-ideal. Since $j(\mathfrak a)\in\ro_L$,
there exists (\cite[\S4.3]{deuring}) a finite extension $L'$ of $L$, a curve $E'$
defined over $L'$ and a prime ideal $\mathfrak B'\subset\ro_{L'}$ lying above $\mathfrak B$
such that $j(E')=j(\mathfrak a)$ and the equation of $E'$ reduced modulo $\mathfrak B'$
gives a non-singular curve (over a finite field).

Since $j(E')$ equals a singular value, $E'$ has complex multiplication by $\ro_D$.
Since $j$-invariant of the reduced curve equals the reduced $j$-invariant (because
$j$-invariant is a rational function in coefficients) and $\mathfrak B'\cap\ro_L=\mathfrak B$,
$j$-invariant of the curve $E'$ reduced modulo $\mathfrak B'$ equals
$R_{\mathfrak B'}(j(\mathfrak a))=R_{\mathfrak B}(j(\mathfrak a))=j(E'')$. Two elliptic curves
are isomorphic if and only if their $j$-invariants are equal (\cite[Proposition III.1.4]{silverman}),
therefore $E''$ is isomorphic to the reduction of $E'$. Finally, Lemma \ref{el.stupid} and
the properties of the reduction (\cite[Theorem 13.12]{lang}) imply that $E''$ is non-supersingular
and $\End(E'')\cong\End(E')\cong\ro_D$.

Therefore, the fourth stage starts with the curve $E''$ defined over $\F_q$, which has complex
multiplication by $\ro_D$. As shown above, the order of the curve $E''$ equals $q+1-u$,
where $4q=u^2+|D|v^2$ and $\gcd(q,u)=1$, but $u$, $v$ are not necessarily equal to $\hat u$, $\hat v$.
Let $\pi=\frac{u+v\sqrt D}2\in\ro_D\subset\ro$. Due to Lemma \ref{el.stupid} there is either
$\pi\ro=\mathfrak p^n$ or $\pi\ro=\overline{\mathfrak p}^n$. The same holds for $\hat\pi\ro$.
Thus, there is either $\pi\ro=\hat\pi\ro$ or $\pi\ro=\overline{\hat\pi}\ro$. Since the norm
of the $\ro_D$-ideal $\pi\ro_D$ equals $|Norm(\pi)|=q$ and is prime to the conductor $f$
of the order $\ro_D$, we have $\pi\ro_D=\pi\ro\cap\ro_D$ (\cite[Proposition 7.20]{cox}).
Similarly, we have $\hat\pi\ro_D=\hat\pi\ro\cap\ro_D$. Thus, there is either $\pi\ro_D=\hat\pi\ro_D$
or $\pi\ro_D=\overline{\hat\pi}\ro_D$. Equivalently, the number $\pi$ is associated either
with $\hat\pi$ or with $\overline{\hat\pi}$ in the ring $\ro_D$.

It is well known (e.g. \cite[Proposition 13.1.5]{ireland}) that the group of units in $\ro_D$
is $\{\pm1\}$ if $D\not\in\{-3,-4\}$, $\{\pm1,\pm\zeta_3,\pm\zeta_3^2\}$ if $D=-3$,
$\{\pm1,\pm i\}$ if $D=-4$. Here $\zeta_3=e^{2\pi i/3}=\frac{-1+\sqrt{-3}}2$.

If $D\not\in\{-3,-4\}$, we have $\pi=\pm\hat\pi$ or $\pi=\pm\overline{\hat\pi}$. This
corresponds to $u=\pm\hat u$. Therefore, in this case $|E''(\F_q)|=q+1\pm\hat u$.
If $|E''(\F_q)|=q+1-\hat u$, the curve $E''$ is the one we search for. Otherwise,
we construct the quadratic twist of $E''$ as follows. If $p\ne2$, the normal Weierstrass
form of the curve equation is $y^2=f(x)$, where $f$ is a polynomial of degree 3 with
the high-order coefficient equal to 1 (in particular, the formulas above give the equation
in this form), and the curve $y^2=c^3f(x/c)$, where $c$ is any quadratic non-residue
in $\F_q$, has the required order (\cite{LayZimmer}). If $p=2$, the normal form is
$y^2+xy=x^3+a_2x^2+a_6$ and the curve $y^2+xy=x^3+(a_2+\gamma)x^2+a_6$, where
$Tr_{\F_q/\F_2}\gamma=1$, has the required order (\cite{LayZimmer}).

If $D=-3$, the procedure for calculating $H_D[j]$ yields the polynomial $H_{-3}[j](x)=x$,
it has the only root $j=0$. The formula \eqref{el.dsquare} in this case is $\left(\frac{-3}p\right)=1$
and implies $p\equiv1\pmod3$, in particular, $p>3$. Any curve of the form $y^2=x^3+b$, $b\ne0$,
has the $j$-invariant equal to 0 (\cite[Proposition A.1.1]{silverman}), and all such curves
are $\overline\F_q$-isomorphic (because they have the same $j$-invariant).

Let $\chi$ be the unique multiplicative character on $\F_q$ of order 2. Let
$S_2(b)=\sum_{x\in\F_q}\chi(x^3+b)$. It is easy to see that the order of the curve $y^2=x^3+b$
is equal to $q+1+S_2(b)$. The equality $p\equiv1\pmod3$ implies $q\equiv1\pmod3$.
According to \cite{shrutka} (the article \cite{shrutka} considers only the case $q=p$,
where $\chi$ is the Legendre symbol, but the arguments can be trivially generalized),
there exist $k,l\in\Z$ such that for any cubic non-residue $c\in\F_q^*$ the equalities
$S_2(1)=2k$, $S_2(c^2)=-k\pm3l$, $S_2(c^{-2})=-k\mp3l$ and $q=k^2+3l^2$ hold. Moreover,
$S_2(b)=\chi(t)S_2(bt^3)$ for any $t\in\F_q^*$.

The curve $E''$ generated in the third stage is $y^2=x^3+1$; therefore, $u=-S_2(1)=-2k$,
$q=k^2+3l^2$, $l=\pm\frac{S_2(c^2)-S_2(c^{-2})}6$ for any cubic non-residue $c$ in $\F_q$.
Since $|\pi|^2=q$, it follows that $\pi=-k\pm l\sqrt{-3}$. Replacing $c$ to $c^{-1}$
if needed, we obtain $\pi=-k-l\sqrt{-3}$, $l=\frac{S_2(c^2)-S_2(c^{-2})}6$.

Either $\hat\pi$ or $\overline{\hat\pi}$ equals the product of $\pi$ and some unit of $\ro_{-3}$.
In both cases $\hat u=2\re\hat\pi$ equals twice the real part of the product of $\pi$ and
some unit. Thus, there are 6 possible variants for $\hat u$:
\begin{itemize}
\item $\hat u=2\re\hat\pi=\pm2\re\pi=\pm2k$. In this case we search for a curve of order
$q+1\pm2k$. One of curves $y^2=x^3+1$ and $y^2=x^3+g^3$, where $g$ is any quadratic non-residue
in $\F_q$, gives the answer.
\item $\hat u=2\re\hat\pi=\pm2\re(\zeta_3)\pi=\pm(k+3l)$. In this case we search for a curve
of order $q+1\pm(k+3l)$. One of curves $y^2=x^3+c^2$ and $y^2=x^3+c^2g^3$ gives the answer.
\item $\hat u=2\re\hat\pi=\pm2\re(\zeta_3^2\pi)=\pm(k-3l)$. In this case we search for a curve
of order $q+1\pm(k-3l)$. One of curves $y^2=x^3+c^{-2}$ and $y^2=x^3+c^{-2}g^3$ gives the answer.
\end{itemize}

If $D=-4$, we similarly have $H_{-4}[j](x)=x-1728$ with the only root $j=1728$.
The formula \eqref{el.dsquare} in this case is $\left(\frac{-4}p\right)=1$
and implies $p\equiv1\pmod4$, in particular, we still have $p>3$. Any curve of the form
$y^2=x^3+bx$, $b\ne0$, has the $j$-invariant equal to 1728 (\cite[Proposition A.1.1]{silverman}),
all such curves are $\overline\F_q$-isomorphic (because they have the same $j$-invariant).

Let $S_1(b)=\sum_{x\in\F_q}\chi(x)\chi(x^2+b)$, where $\chi$ is the unique multiplicative
character on $\F_q$ of order 2, as above. It is easy to see that the order of the curve
$y^2=x^3+bx$ equals $q+1+S_1(b)$. The equality $p\equiv1\pmod4$ implies $q\equiv1\pmod4$.
According to \cite{jacobstahl}, there exist $k,l\in\Z$ such that $k$ is odd,
$S_1(1)=2k$, $S_1(b)=\pm2l$ for any quadratic non-residue $b$ and $S_1(b)=\chi(t)S_1(bt^2)$
for any $t\in\F_q^*$.

The curve $E''$ generated in the third stage is $y^2=x^3+x$; therefore, $u=-S_1(1)=-2k$,
$q=k^2+l^2$. Since $|\pi|^2=q$, it follows that $\pi=-k\pm li$.

Similarly to the previous case, there are 4 possible variants for $\hat u$: $\pm2\re\pi=\pm2k$
and $\pm2\re(i\pi)=\pm2l$. If $\hat u=\pm2k$, one of curves $y^2=x^3+x$ and $y^2=x^3+g^2x$,
where $g$ is any quadratic non-residue in $\F_q$, has the required order. If $\hat u=\pm2l$,
one of curves $y^2=x^3+gx$ and $y^2=x^3+g^3x$ has the required order.

\subsection{Some known optimizations}\label{el.otherinv}
The coefficients of the polynomial $H_D[j]$ grow quite fast with $|D|$. For example,
$H_{-40}[j](x)=x^2-425692800x+9103145472000$. Consequently, it is useful
to search for another functions with singular values in $L$, which have a smaller
height of the characteristic polynomial.

Let $z\in\h$, $q=e^{2\pi iz}$. Let us introduce some functions following \cite{weber}:
$$
\eta(z)=q^{\frac1{24}}\prod_{n=1}^\infty(1-q^n)=q^{\frac1{24}}\sum_{n=-\infty}^\infty(-1)^n q^{\frac{3n^2+n}2},
$$
$$
\mathfrak f(z)=e^{-\frac{\pi i}{24}}\frac{\eta\left(\frac{z+1}2\right)}{\eta(z)},\quad
\mathfrak f_1(z)=\frac{\eta\left(\frac z2\right)}{\eta(z)},\quad
\mathfrak f_2(z)=\sqrt2\frac{\eta(2z)}{\eta(z)},
$$
\begin{equation}\label{el.gammadef}
\gamma_2(z)=\frac{\mathfrak f^{24}-16}{\mathfrak f^8}=\frac{\mathfrak f_1^{24}+16}{\mathfrak f_1^8}=\frac{\mathfrak f_2^{24}+16}{\mathfrak f_2^8},
\quad j(z)=\gamma_2(z)^3.
\end{equation}

Let $N$ be a natural number. We define a \textit{$N$-system} (following \cite{schertz}) as a set of
forms $(A_1,B_1,C_1)$, \dots, $(A_h,B_h,C_h)$ such that
\begin{itemize}
\item the set $\{\mathfrak h(A_i,B_i,C_i):1\le i\le h\}$ is the complete system of representatives
of the group $\mathcal H_D$,
\item the relations $$\gcd(A_i,N)=1;\quad B_i\equiv B_j\pmod{2N}$$ hold.
\end{itemize}
Note that for any form $(A_i,B_i,C_i)$ the congruence $B_i\equiv D\pmod2$ is true, so the first
condition implies that $B_i\equiv B_j\pmod 2$ for any $i,j$.

If a set of forms satisfying the first condition is known, it is easy to construct a $N$-system.
For example, the complete set of reduced forms can be used as a starting point. The corresponding
algorithm can be found in \cite[proof of Proposition 3]{schertz}. (We assume that the prime factorization
of $N$ is known.)
\begin{enumerate}
\item First, achieve the condition $\gcd(A_i,N)=1$ for all $i$.

Obviously, it is sufficient to solve the next task: achieve $\gcd(A_i,N_0l)=1$
assuming that $\gcd(A_i,N_0)=1$, where $l$ is the next prime divisor of $N$ not dividing $N_0$.

The number $l$ can not divide all of numbers $A_i$, $A_i+N_0B_i+N_0^2C_i$, $l^2A_i+lN_0B_i+N_0^2C_i$,
because otherwise the numbers $A_i$, $B_i$, $C_i$ would have the common divisor $l$ and the form
$(A_i,B_i,C_i)$ would not be primitive.
\begin{itemize}
\item Assume $l\nmid A_i$. Then the condition $\gcd(A_i,N_0l)$ already holds.
\item Assume $l\nmid A_i+N_0B_i+N_0^2C_i$. Change the variables $x=x'$, $y=N_0x'+y'$
and replace the current form with the new form (obviously, it is equivalent).
\item Assume $l\nmid l^2A_i+lN_0B_i+N_0^2C_i$. Find $a,b\in\Z$ such that $al-bN_0=1$,
change the variables $x=lx'+by'$, $y=N_0x'+ay'$ and replace the current form
with the new form (obviously, it is equivalent).
\end{itemize}
\item Next, achieve the condition $B_i\equiv B_1\pmod{2N}$ for all $i$.
The change of the variables $x=x'+ay'$, $y=y'$ transforms the form
$(A_i,B_i,C_i)$ to the equivalent form $(A_i,B_i+2aA_i,C_i+aB_i+a^2A_i)$.
Since $\gcd(A_i,N)=1$, it is sufficient to apply this transformation
with $a=A_i^{-1}\frac{B_1-B_i}2\bmod N$.
\end{enumerate}

\begin{theorem}\label{el.t1} (\cite[Theorem 1]{schertz}) Let $\alpha\in\h$ be the root of the form
$$(A,B,C),\quad 2\nmid A,\quad 32\mid B,$$
with the discriminant $B^2-4AC=D=-4m$, $m\in\N$. Let $g(\alpha)$ be defined by the
following formulas:
$$
\begin{aligned}
\displaystyle\left(\left(\frac2A\right)\frac1{\sqrt2}\mathfrak f(\alpha)^2\right)^3, \mbox{ if }m\equiv1\pmod8,\\
\displaystyle\mathfrak f(\alpha)^3, \mbox{ if }m\equiv3\pmod8,\\
\displaystyle\left(\frac12\mathfrak f(\alpha)^4\right)^3, \mbox{ if }m\equiv5\pmod8,\\
\displaystyle\left(\left(\frac2A\right)\frac1{\sqrt2}\mathfrak f(\alpha)\right)^3, \mbox{ if }m\equiv7\pmod8,\\
\displaystyle\left(\left(\frac2A\right)\frac1{\sqrt2}\mathfrak f_1(\alpha)^2\right)^3, \mbox{ if }m\equiv2\pmod4,\\
\displaystyle\left(\left(\frac2A\right)\frac1{2\sqrt2}\mathfrak f_1(\alpha)^4\right)^3, \mbox{ if }m\equiv4\pmod8.
\end{aligned}
$$
Then $g(\alpha)\in\ro_L$.

If $\alpha_1=\alpha,\dots,\alpha_h$ are roots of the elements of 16-system, the singular values $g(\alpha_i)$
form the complete set of different conjugates over $\Q$.
\end{theorem}

\begin{theorem}\label{el.t2} (\cite[Theorem 2]{schertz}) Let $\alpha\in\h$ be the root of the form
$$(A,B,C),\quad 3\nmid A,\quad 3\mid B,$$
with the discriminant $B^2-4AC=D$. Then
$$
\Q(\gamma_2(\alpha))=\begin{cases}
\Q(j(\alpha)),&3\nmid D,\\
\Q(j(3\alpha)),&3\mid D.
\end{cases}
$$
Moreover, if $3\nmid D$ and $\alpha_1=\alpha,\dots,\alpha_h$ are roots of the elements of 3-system,
the singular values $\gamma_2(\alpha_i)$ form the complete set of different conjugates over $\Q$.
In addition, $\gamma_2(\alpha_i)$ are algebraic integers.
\end{theorem}

Let $p_1$, $p_2$ be prime numbers. Following \cite{doubleeta}, we introduce the function
$$
\mathfrak m_{p_1,p_2}(z)=\frac{\eta\left(\frac z{p_1}\right)\eta\left(\frac z{p_2}\right)}{\eta(z)\eta\left(\frac z{p_1p_2}\right)}
$$
and define $s=\frac{24}{\gcd(24,(p_1-1)(p_2-1))}$.

\begin{theorem}\label{el.t3} (\cite[Theorems 3.2, 3.3, Corollary 3.1]{doubleeta})
Let $D$ satisfy \eqref{el.dbase}, $N=p_1p_2$, $p_1$ and $p_2$ are primes such that
\begin{enumerate}
\item[] 1) $\left(\frac D{p_1}\right),\left(\frac{D}{p_2}\right)\ne-1$ if $p_1\ne p_2$;
\item[] either 2a) $\left(\frac Dp\right)=1$ if $p_1=p_2=p$, or 2b) $p|f$ if $p_1=p_2=p$.
\end{enumerate}
Then there exists a form $(A_1,B_1,C_1)$ such that $\gcd(A_1,N)=1$ and $N\mid C_1$.
Let $\alpha_1\in\h$ be the root of this form. The singular value $\mathfrak m_{p_1,p_2}^s(\alpha_1)$ lies in $L$.
All conjugates over $K$ to $\mathfrak m_{p_1,p_2}^s(\alpha_1)$ are $\mathfrak m_{p_1,p_2}^s(\alpha_i)$,
where $\alpha_i$ are roots of elements of $N$-system. The numbers $\mathfrak m_{p_1,p_2}^s(\alpha_i)$
are algebraic integers.

If one of conditions 1) and 2a) holds, the numbers $\mathfrak m_{p_1,p_2}^s(\alpha_i)$ are units
(i.e. the numbers $\mathfrak m_{p_1,p_2}^{-s}(\alpha_i)$ are algebraic integers too).

If primes $p_1$ and $p_2$ satisfy the stronger condition:
\begin{itemize}
\item $\left(\frac D{p_1}\right), \left(\frac D{p_2}\right)\ne-1$ and $p_1,p_2\nmid f$ when $p_1\ne p_2$;
\item $\left(\frac Dp\right)=1$ or $p\mid f$ when $p_1=p_2=p\ne2$;
\item either $\left(\frac D2\right)=1$, or $2\mid f$ and $D\not\equiv4\pmod{32}$ when $p_1=p_2=2$,
\end{itemize}
then the complex conjugation rearranges $\mathfrak m_{p_1,p_2}^s(\alpha_i)$.
\end{theorem}

We need more precise statements for the following. The formulations of theorems \ref{el.t1}, \ref{el.t2},
\ref{el.t3} do not give the full information regarding the action of $\Gal(L/K)$ on singular values.
However, the proofs from the articles \cite{schertz} and \cite{doubleeta} contain this information.

\begin{statement}\label{el.galaction} (\cite[Theorem 7]{schertz}) Let $\theta$ be the function
from one of theorems \ref{el.t1}, \ref{el.t2}, \ref{el.t3} (in the last one the weak condition
on $p_1$, $p_2$ is sufficient). Let $\mathfrak a$, $\mathfrak b$ be two elements of the $N$-system
from the same theorem. Let $\alpha$ be the root of $\mathfrak a$, $\beta$ be the root of $\mathfrak b$,
$\Omega:\mathcal H_D\to\Gal(L/K)$ be the canonical isomorphism. Then
$$
\theta(\alpha)^{\Omega(\mathfrak h(\mathfrak a)\mathfrak h(\mathfrak b)^{-1})}=\theta(\beta).
$$
\end{statement}
This formula holds for $\theta=j$ too, as mentioned above.

It is more convenient to use Statement \ref{el.galaction} in the form of a formula
which specifies the action of a given automorphism from $\Gal(L/K)$ on a given singular
value. We remind that $\mathfrak h$ is surjective and $N$-system contains representatives
of all classes in $\mathcal H_D$.

\begin{corollary}\label{el.galaction.conv}
Let $\theta$, $N$-system and $\mathfrak a$ be the same as in Statement \ref{el.galaction}.
Let $\mathfrak c\in\mathcal H_D$. Then there exists a form $\mathfrak b$ from the $N$-system such that
$$
\mathfrak h(\mathfrak b)=\mathfrak h(\mathfrak a)\mathfrak c^{-1}.
$$
If $\beta$ is the root of $\mathfrak b$, then
\begin{equation}\label{el.galaction2}
\theta(\alpha)^{\Omega(\mathfrak c)}=\theta(\beta).
\end{equation}
\end{corollary}

Further for the function $\mathfrak m_{p_1,p_2}^s$ we assume that $p_1$ and $p_2$
satisfy the strong condition of the theorem; it is easy to see that such primes
can be found for any $D$.

Theorems \ref{el.t1} and \ref{el.t2} can be joined: if the discriminant $D$
and the form $(A,B,C)$ satisfy the assumptions of Theorem \ref{el.t1}
and also $3\nmid D$, $3\nmid A$, $3\mid B$, then the function $g(\alpha)$
can be defined without the exponent $3$ and the consequence of Theorem \ref{el.t1}
still holds. For example, let us consider the case $m\equiv3\pmod8$.
According to \eqref{el.gammadef},
\begin{equation}\label{el.t1plust2}
\mathfrak f(\alpha)=\frac{\left(\mathfrak f(\alpha)^3\right)^3\gamma_2(\alpha)}
{\left(\mathfrak f(\alpha)^3\right)^8-16}.
\end{equation}
Since $\mathfrak f(\alpha)^3\in L$ and $\gamma_2(\alpha)\in L$,
we have $\mathfrak f(\alpha)\in L$. Statement \ref{el.galaction} now implies
that any automorphism from $\Gal(L/K)$ maps $\mathfrak f(\alpha)^3$ to
$\mathfrak f(\alpha')^3$ and $\gamma_2(\alpha)$ to $\gamma_2(\alpha')$,
where $\alpha'$ depends only on the automorphism; \eqref{el.t1plust2}
implies that $\mathfrak f(\alpha)$ is mapped to $\mathfrak f(\alpha')$.
Finally, $\mathfrak f(\alpha)$ is an algebraic integer e.g. as a cubic root
from $\mathfrak f(\alpha)^3$ which is an algebraic integer due to Theorem \ref{el.t1}.
In other cases formulas are slightly more complicated, but the reasoning is the same.

Let $\theta$ and $\alpha_*=\{\alpha_1,\dots,\alpha_h\}$ be the function and the set
of roots from one of theorems \ref{el.t1}--\ref{el.t3}. Let us consider
the polynomial in one variable
$$
H_D[\theta,\alpha_*](x)=\prod_{i=1}^h(x-\theta(\alpha_i)).
$$
This polynomial has integer coefficients. For functions from Theorems \ref{el.t1}
and \ref{el.t2} this follows directly from the consequence of theorem.
For $\theta=\mathfrak m_{p_1,p_2}^s$ it is easy to see from Statement
\ref{el.galaction} that $H_D[\theta,\alpha_*]$ is invariant under $\Gal(L/K)$
and therefore is in $K[x]$ and it remains to apply Theorem \ref{el.t3}.

For example, $H_{-40}[\gamma_2,\alpha_*](x) = x^2-780x+20880$,
$H_{-40}[g,\alpha_*](x)=x^2-x-1$ with
$g(\alpha)=\left(\frac2A\right)\frac1{\sqrt2}\mathfrak f_1(\alpha)^2$,
$H_{-40}[\mathfrak m_{5,7},\alpha_*](x)=x^2-x-1$,
$H_{-40}[\mathfrak m_{11,13},\alpha_*](x)=
x^2\pm2x+1$. (The choice of $\alpha_*$ does not affect the polynomial
in first three cases; there are two variants for the polynomial depending
on $\alpha_*$ in the last case.) The last example shows that the values
of $\mathfrak m_{p_1,p_2}^s(\alpha_i)$ can coincide, so in the general case
$H_D[\mathfrak m_{p_1,p_2}^s,\alpha_*]$ is some power of the minimal
polynomial.

Since the polynomial $H_D[\theta,\alpha_*]$ has integer coefficients similar to
$H_D[j]$, it also can be calculated by calculating sufficiently accurate
approximations to the singular values $\theta(\alpha_i)$, multiplying factors
$x-\theta(\alpha_i)$ and rounding coefficients to integers. Since $\theta(\alpha_i)\in\ro_L$,
it has a representative in $\ro_L/\mathfrak B\subset\F_q$ (Theorem \ref{el.subsetinq}),
so the reduction of $H_D[\theta,\alpha_*](x)$ modulo $p$ splits into
linear factors in $\F_q$. It remains to calculate the $j$-invariant
by the reduction of $\theta(\alpha_i)$ in $\F_q$. The formulas \eqref{el.gammadef}
give the answer for $\theta=\gamma_2$ and $\theta$ being a power of $\mathfrak f$
from Theorem \ref{el.t1}. The situation for $\theta=\mathfrak m_{p_1,p_2}^s$
is more complicated. There exists the polynomial $\Phi_{p_1,p_2}(x,y)\in\Z[x,y]$
such that the identity $\Phi_{p_1,p_2}(\mathfrak m_{p_1,p_2}^s(z),j(z))=0$ holds
(\cite{doubleeta}). Substituting $z=\alpha_i$ and reducing modulo $\mathfrak B$
(since $\mathfrak B\cap\Z=p\Z$, it is sufficient to reduce $\Phi_{p_1,p_2}$ modulo $p$),
we obtain a polynomial equation for the required $j$-invariant. Solving this equation
gives several variants for the $j$-invariant. The correct one can be selected e.g.
as follows: construct an elliptic curve (and its quadratic twist) for every variant
and check whether its order equals $q+1-\hat u$. For example, cryptographic applications
require that $q+1-\hat u$ has a large prime divisor; in this case a simple test
$(q+1-\hat u)P=0$ for a random point $P$ is good for eliminating wrong candidates.
Note that the right order does not guarantee that the endomorphism ring is precisely $\ro_D$,
but such a subtle difference is usually not important; more detailed discussion can
be found in \cite{doubleeta}.

\section{Properties of the isomorphism $\Omega$}\label{omegaprop}
We recall that the group $\mathcal H_D$ is the factorgroup of the group
$I(\ro_D)$ of proper fractional $\ro_D$-ideals by the subgroup $P(\ro_D)$
of principal ideals.

An $\ro_D$-ideal $\mathfrak a$ is \textit{prime to} $f$ when
$\mathfrak a+f\ro_D=\ro_D$. This is equivalent to $\gcd(Norm(\mathfrak a),f)=1$,
and every ideal prime to the conductor is proper (\cite[Lemma 7.18]{cox}). Let $I(\ro_D,f)$ denote
the subgroup in $I(\ro_D)$ generated by ideals prime to $f$.
Let $P(\ro_D,f)$ denote the subgroup in $P(\ro_D)$ generated by principal ideals
$\alpha\ro_D$ with $\gcd(Norm(\alpha),f)=1$. The inclusion
$I(\ro_D,f)\subset I(\ro_D)$ induces an isomorphism $I(\ro_D,f)/P(\ro_D,f)\cong\mathcal H_D$
(\cite[Proposition 7.19]{cox}).

An $\ro$-ideal $\mathfrak a$ is prime to $f$ if and only if $\gcd(Norm(\mathfrak a),f)=1$
(\cite[Lemma 7.18]{cox}). Let $I(\ro,f)$ denote the subgroup of fractional $\ro$-ideals
generated by ideals prime to $f$. We recall that $P_{K,\Z}(f)$ denotes the subgroup
of $\ro$-ideals generated by principal ideals of the form $\alpha\ro$ with
$\alpha\in\ro$, $\alpha\equiv a\pmod{f\ro}$ for some $a\in\Z$, $\gcd(a,f)=1$.
The map $\Omega_1:\mathfrak a\mapsto\mathfrak a\ro$ gives a group isomorphism
$I(\ro_D,f)\to I(\ro,f)$ which preserves the norm (\cite[Proposition 7.20]{cox}).
In addition (\cite[Proposition 7.22]{cox}), $\Omega_1$ induces an isomorphism
$I(\ro_D,f)/P(\ro_D,f)\cong I(\ro,f)/P_{K,\Z}(f)$.

Thus, we have an isomorphism $\Omega_2:\mathcal H_D\to I(\ro,f)/P_{K,\Z}(f)$.
The Artin map $I(\ro,f)\to\Gal(L/K)$ (denoted as $\left(\frac{L/K}{\cdot}\right)$)
induces an isomorphism $I(\ro,f)/P_{K,\Z}(f)\to\Gal(L/K)$. The composition of
the last isomorphism with $\Omega_2$ is the canonical isomorphism $\Omega$
referenced in Statement \ref{el.galaction} (\cite[\S9]{cox}).

Let us sum up the above maps. There exists a commutative diagram
\begin{center}
\hfill
\parbox{420pt}{\begin{picture}(420,70)
\put(0,60){\makebox[42pt]{\hbox to 42pt{\hfil $I(\ro_D)$ \hfil}}}
\put(20,55){\vector(0,-1){11}}
\put(45,60){\makebox[20pt]{\hbox to 20pt{\hfil $\supset$ \hfil}}}
\put(65,60){\makebox[125pt]{\hbox to 125pt{\hfil $I(\ro_D,f)$ \hfil}}}
\put(125,55){\vector(0,-1){11}}
\put(170,63){\vector(1,0){55}}
\put(185,67){\makebox[30pt]{\hbox to 30pt{\hfil $\scriptstyle\Omega_1$ \hfil}}}
\put(215,60){\makebox[105pt]{\hbox to 105pt{\hfil $I(\ro,f)$ \hfil}}}
\put(262,55){\vector(0,-1){11}}
\put(0,30){\makebox[42pt]{\hbox to 42pt{\hfil $\mathcal H_D$ \hfil}}}
\put(40,33){\vector(1,0){20}}
\put(65,30){\makebox[125pt]{\hbox to 125pt{\hfil $I(\ro_D,f)/P(\ro_D,f)$ \hfil}}}
\put(190,33){\vector(1,0){20}}
\put(215,30){\makebox[105pt]{\hbox to 105pt{\hfil $I(\ro,f)/P_{K,\Z}(f)$ \hfil}}}
\put(320,33){\vector(1,0){20}}
\put(300,55){\vector(3,-1){45}}
\put(320,52){$\scriptstyle\left(\frac{L/K}{\cdot}\right)$}
\put(345,30){$\Gal(L/K)$}
\qbezier(30,25)(120,0)(215,25)
\put(215,25){\vector(4,1){0}}
\put(120,15){$\scriptstyle\Omega_2$}
\qbezier(25,25)(185,-15)(345,25)
\put(345,25){\vector(4,1){0}}
\put(185,7){$\scriptstyle\Omega$}
\end{picture}}
\hfill
\refstepcounter{equation}\label{el.omegadef}\llap{(\theequation)}
\end{center}
where vertical arrows denote projections of a group to its factorgroup
and horizontal arrows in the second line are isomorphisms.

\begin{theorem}\label{el.artinaction}
Let $(A,B,C)$ be a form with $\gcd(A,D)=1$. Let $q\mid D$ be an integer satisfying
one of the conditions:
\begin{itemize}
\item $|q|$ is an odd prime, $q\equiv1\pmod4$; or
\item $q\in\{-4,\pm8\}$, $\frac Dq\equiv0\pmod4$ or $\frac Dq\equiv1\pmod4$.
\end{itemize}
Then
\begin{enumerate}
\item $\sqrt q\in L$.
\item $\mathfrak a=\langle A,\frac{-B+\sqrt D}2\rangle_\Z\in I(\ro_D,f)$, $Norm(\mathfrak a)=A$.
\item $$\left(\frac{L/K}{\Omega_1(\mathfrak a)}\right)(\sqrt q)=\left(\frac qA\right)\sqrt q.$$
\end{enumerate}
\end{theorem}
\begin{proof}
The first assertion follows from \cite[Theorem 2.2.23 and (2.2.8)]{tsss}.

\cite[Theorem 7.7]{cox} implies that $\mathfrak a$ is a proper $\ro_D$-ideal.
Its norm is $|\ro_D/\mathfrak a|$ by definition; it is easy to see that
every coset in $\ro_D/\mathfrak a$ contains exactly one integer from
$0,\dots,A-1$, so $Norm(\mathfrak a)=A$. Since $\gcd(A,f)=1$, the ideal $\mathfrak a$
is prime to $f$. The second assertion is proved.

Let $\Omega_1(\mathfrak a)=\mathfrak p_1\dots\mathfrak p_s$, where $\mathfrak p_i$ are
prime $\ro$-ideals (not necessarily different). Since
$$A=Norm(\mathfrak a)=Norm(\mathfrak p_1)\dots Norm(\mathfrak p_s)$$
and the Kronecker symbol is multiplicative, it is sufficient to prove that
for every prime ideal $\mathfrak p$ dividing $\Omega_1(\mathfrak a)$ the equality with the Artin symbol
\begin{equation}\label{el.artin1}
\left(\frac{L/K}{\mathfrak p}\right)(\sqrt{q})=\left(\frac{q}{Norm(\mathfrak p)}\right)\sqrt{q}.
\end{equation}
holds. The left-hand side is an image of $\sqrt{q}$ under an automorphism, so it must be one of
$\pm\sqrt q$.

Assume first that $\mathfrak p\mid\Omega_1(\mathfrak a)$, $\mathfrak p\cap\Z=p\Z$, $p$ is odd.
Let $\mathfrak B$ be a prime $\ro_L$-ideal lying above $\mathfrak p$. Since $\gcd(A,D)=1$ and
$q\mid D$, we have $2\sqrt q\not\in\mathfrak B$ and therefore $\sqrt q\not\equiv-\sqrt q
\pmod{\mathfrak B}$. By definition
$$\left(\frac{L/K}{\mathfrak p}\right)(\sqrt{q})\equiv\sqrt{q}^{Norm(\mathfrak p)}
=q^{\frac{Norm(\mathfrak p)-1}2}\sqrt{q}\pmod{\mathfrak B}.$$

If the ideal $p\ro$ is prime (i.e. $\mathfrak p=p\ro$), then $Norm(\mathfrak p)=p^2$
and the right-hand side of \eqref{el.artin1} equals $\sqrt q$. On the other part,
$q^{\frac{Norm(\mathfrak p)-1}2}=(q^{p-1})^{\frac{p+1}2}\equiv1\pmod p$, so
the left-hand side of \eqref{el.artin1} is congruent to $\sqrt q$ modulo $\mathfrak B$
and therefore is equal to $\sqrt q$. Thus, \eqref{el.artin1} is proved in this case.

If the ideal $p\ro$ is not prime, then $Norm(\mathfrak p)=p$ and the right-hand side
of \eqref{el.artin1} equals $\left(\frac qp\right)\sqrt q$. On the other part,
$q^{\frac{Norm(\mathfrak p)-1}2}=q^{\frac{p-1}2}\equiv\left(\frac qp\right)\pmod p$,
so the left-hand side of \eqref{el.artin1} is congruent to $\left(\frac qp\right)\sqrt q$
modulo $\mathfrak B$ and therefore is equal to $\left(\frac qp\right)\sqrt q$. Thus,
\eqref{el.artin1} is proved in this case too.

Assume now that $\mathfrak p\mid\Omega_1(\mathfrak a)$, $\mathfrak p\cap\Z=2\Z$,
a prime $\ro_L$-ideal $\mathfrak B$ lies above $\mathfrak p$. In this case $2\mid A$,
the assumption of theorem implies that $2\nmid D$ and $q$ is odd. Since
$B^2-4AC=D$, we have $D\equiv B^2\equiv1\pmod8$. Thus $d\equiv1\pmod8$
and the ideal $2\ro$ is not prime (\cite[Proposition 13.1.4]{ireland}), so
$Norm(\mathfrak p)=2$. Therefore, the right-hand side of \eqref{el.artin1} equals
$\left(\frac q2\right)\sqrt q$. To calculate the left-hand side of \eqref{el.artin1},
consider $$\left(\frac{L/K}{\mathfrak p}\right)\left(\frac{1+\sqrt q}2\right).$$
This expression must be equal to one of $\frac{1\pm\sqrt q}2$, two possible values
are different modulo $\mathfrak B$. By definition
$$
\left(\frac{L/K}{\mathfrak p}\right)\left(\frac{1+\sqrt{q}}2\right)\equiv
\left(\frac{1+\sqrt{q}}2\right)^2=\frac{q-1}4+\frac{1+\sqrt{q}}2\pmod{\mathfrak B}.
$$
If $\left(\frac q2\right)=1$, then $q\equiv1\pmod8$, $\frac{q-1}4$ is even and hence
lies in $\mathfrak B$. If $\left(\frac q2\right)=-1$, then $q\equiv5\pmod8$, $\frac{q-1}4$
is odd and therefore is congruent to $-1\equiv1$ modulo $\mathfrak B$. In both cases
there is
$$
\left(\frac{L/K}{\mathfrak p}\right)\left(\frac{1+\sqrt{q}}2\right)\equiv
\frac{1+\left(\frac{q}2\right)\sqrt{q}}2\pmod{\mathfrak B},
$$
which implies \eqref{el.artin1}.
\end{proof}

\begin{lemma}\label{el.qdef} Let $d<0$ satisfy one of conditions \eqref{eq.d1} and
\eqref{eq.d2}. There exists the unique (up to the order of factors) representation of $d$
as the product $$d=q_1^*\dots q_t^*,$$ where all $q_i^*$ are pairwise relatively prime,
$$q^*=(-1)^{\frac{q-1}2}q,$$
if $q>0$ is an odd prime, and $q^*\in\{-4,\pm8\}$ if $q=2$.
\end{lemma}
\begin{proof}
The uniqueness is obvious, we need to prove the existence.

If $d$ satisfies \eqref{eq.d2}, the prime factorization of $d$ has the form $d=-q_1\dots q_t$,
where $q_i$ are different odd primes; since $q_i^*=\pm q_i$, it follows that $d=\pm q_1^*\dots q_t^*$;
finally, the sign is correct due to $d\equiv1\pmod4$ and $q_i^*\equiv1\pmod4$ for all $i$.

Assume that $d$ satisfies \eqref{eq.d1}. The prime factorization of $\frac d4$ has one of the forms
$\frac d4=-q_1\dots q_{t-1}$ or $\frac d4=-2q_1\dots q_{t-1}$, where $q_i$ are different odd primes
in both forms. If $\frac d4$ is odd, similarly to the previous case we obtain
$\frac d4=\pm q_1^*\dots q_{t-1}^*$, but this time \eqref{eq.d1} implies $\frac d4\not\equiv1\pmod4$,
so the sign is "-". Multiplying by 4, we obtain the assertion of the lemma. Finally,
if $\frac d4$ is even, we have $\frac d4=\pm2q_1^*\dots q_{t-1}^*$. Selecting the correct sign
in $q_t^*=\pm8$, we obtain the assertion of the lemma.
\end{proof}

It is easy to see that the numbers $q_i^*$ from Lemma \ref{el.qdef} satisfy the assumptions of
Theorem \ref{el.artinaction}. Therefore, $K\qtsqrt\subset L$. The field $K\qtsqrt$ depends only
on the field $K$ (which defines $d$ but not $f$) and is called the \textit{genus field} for $K$.
Hereafter we use the notation
$$
K\qtsqrt=K_G.
$$

\section{Ring of algebraic integers in the genus field}\label{el.integers}
Let $q_i^*$ be as in Lemma \ref{el.qdef}. There are three cases.
\begin{enumerate}
\item All $|q_i|$ are odd primes.
\item $q_t^*=\pm8$.
\item $q_t^*=-4$.
\end{enumerate}
We need to know a basis of algebraic integers in the field $K_G$ over $\Z$.
Since $d=q_1^*\dots q_t^*$, we have $\sqrt{d}\in\Q\qtsqrt$ and
therefore $K_G=\Q\qtsqrt$.
The formulas are slightly
different in different cases, so we consider each case separately.

\begin{lemma}\label{el.ramint}
Let $M$ be a number field. Let $p\in\Z$ be a prime such that
the ideal $p\Z$ is unramified in $M$. Let $c\in M$ satisfy
the condition $pc^2\in\ro_M$. Then $c\in\ro_M$.
\end{lemma}
\begin{proof}
Assume that $c\not\in\ro_M$. The fractional ideal $c\ro_M$
has the factorization to the prime ideals $c\ro_M=\mathfrak q_1^{s_1}\dots\mathfrak q_m^{s_m}$,
where $\mathfrak q_i$ are pairwise different and $s_1<0$. The degree
of $\mathfrak q_1$ in the prime factorization of $p\ro_M$ is at most 1 because
$p\ro_M$ is unramified. The degree of $\mathfrak q_1$ in the prime factorization
of $c^2\ro_M$ is at most $-2$. Therefore, the degree of $\mathfrak q_1$
in the prime factorization of $pc^2\ro_M$ is negative. The contradiction with
$pc^2\in\ro_M$ proves the lemma.
\end{proof}

\begin{theorem}\label{el.fundbasisstart}
Let $\tilde q_1,\dots,\tilde q_r$ be pairwise different integers such that
$|\tilde q_i|$ are odd primes and $\tilde q_i\equiv1\pmod4$.
Let $\alpha_i=\frac{1+\sqrt{\tilde q_i}}2$ and $\tilde\alpha_i=\frac{1-\sqrt{\tilde q_i}}2$.
Then:
\begin{enumerate}
\item The set $\{\alpha_1^{s_1}\dots\alpha_r^{s_r}:(s_1,\dots,s_r)\in\{0,1\}^r\}$
is a basis of integers in the field $\Q\tqrsqrt$ over $\Z$.
\item The set $\{\tilde\alpha_1^{s_1}\alpha_1^{1-s_1}\dots\tilde\alpha_r^{s_r}\alpha_r^{1-s_r}:(s_1,\dots,s_r)\in\{0,1\}^r\}$
is a basis of integers in the field $\Q\tqrsqrt$ over $\Z$.
\end{enumerate}
\end{theorem}
\begin{proof}
We prove the theorem by induction on $r$. For $r=0$ the theorem is trivial. Assume
that the theorem is proved for all fields $M_i=\Q\tqsqrt{i}$ with $i=1,\dots,r-1$.

\begin{lemma}\label{el.unramify}
Let $p\in\Z$ be a prime not dividing any of numbers $\tilde q_1,\dots,\tilde q_{r-1}$.
Then the ideal $p\Z$ is unramified in $M_{r-1}$.
\end{lemma}
\begin{proof}
It is sufficient to check that any prime ideal of the field $M_{i-1}$ dividing
$p\ro_{M_{i-1}}$ is unramified in $M_i=M_{i-1}(\sqrt{\tilde q_i})$ for all $1\le i\le r-1$.

Let $\mathfrak p$ be a prime ideal of the field $M_{i-1}$ such that $\mathfrak p\cap\Z=p\Z$.
The extension $M_{i-1}\subset M_i$ is generated by $\alpha_i$; the inductive hypothesis
implies that $(1,\alpha_i)$ is a basis of $\ro_{M_i}/\ro_{M_{i-1}}$. The only nontrivial
automorphism in $\Gal(M_i/M_{i-1})$ transforms this basis to $(1,\tilde\alpha_i)$.
According to \cite[Propositions III.8 and III.14]{langalg}, $\mathfrak p$ is unramified
if $\mathfrak p$ does not divide
$\det\begin{pmatrix}1&\alpha_i\\1&\tilde\alpha_i\end{pmatrix}^2=(\alpha_i-\tilde\alpha_i)^2=\tilde q_i$.
This is true, because $p$ does not divide $\tilde q_i$.
\end{proof}

Apply Lemma \ref{el.unramify} to $p=|\tilde q_r|$. The factorization of $\tilde q_r\ro_{M_{r-1}}$
in the prime ideals does not contain squares. In particular, $\sqrt{\tilde q_r}\not\in M_{r-1}$
because otherwise $\tilde q_r\ro_{M_{r-1}}=(\sqrt{\tilde q_r}\ro_{M_{r-1}})^2$. Therefore,
$(1,\alpha_r)$ is a $M_{r-1}$-basis of $M_r$.

Let $a+b\alpha_r$ be an algebraic integer and $a,b\in M_{r-1}$. The number $a+b(1-\alpha_r)$
is conjugate to $a+b\alpha_r$ and hence is also an algebraic integer. Thus, their
sum $x=2a+b$ and product $y=a^2+ab+b^2\frac{1-\tilde q_r}4$ are also algebraic integers and
lie in $\ro_{M_{r-1}}$. Furthermore, $x^2-4y=\tilde q_r b^2\in\ro_{M_{r-1}}$. Lemma \ref{el.ramint}
implies that $b\in\ro_{M_{r-1}}$. Thus, $2a\in\ro_{M_{r-1}}$, $a^2+ab\in\ro_{M_{r-1}}$,
$2a^2=2(a^2+ab)-2a\cdot b\in\ro_{M_{r-1}}$. Applying Lemmas \ref{el.unramify} and \ref{el.ramint}
to $p=2$, we obtain $a\in\ro_{M_{r-1}}$. So if $a+b\alpha_r$ is an algebraic integer and
$a,b\in M_{r-1}$, then $a,b\in\ro_{M_{r-1}}$. The converse assertion is obvious, so
$(1,\alpha_r)$ is a $\ro_{M_{r-1}}$-basis of $\ro_{M_r}$. This proves the inductive step
for the set $\{\alpha_1^{s_1}\dots\alpha_r^{s_r}\}$. To prove the second assertion
of the theorem it is sufficient to note that
$(1-\alpha_r,\alpha_r)=\left(\frac12(1-\sqrt{\tilde q_r}),\frac12(1+\sqrt{\tilde q_r})\right)$
also is a $\ro_{M_{r-1}}$-basis of $\ro_{M_r}$.
\end{proof}

\begin{theorem}
Let $\tilde q_1,\dots,\tilde q_{r-1}$ be the same as in Theorem \ref{el.fundbasisstart}
and $\tilde q_r=\pm8$. Let $\alpha_r=\sqrt{\frac{\tilde q_r}4}$. Then:
\begin{enumerate}
\item The set $\{\alpha_1^{s_1}\dots\alpha_r^{s_r}:(s_1,\dots,s_r)\in\{0,1\}^r\}$
is a basis of integers in the field $\Q\tqrsqrt$ over $\Z$.
\item The set $$\{\tilde\alpha_1^{s_1}\alpha_1^{1-s_1}\dots\tilde\alpha_{r-1}^{s_{r-1}}\alpha_{r-1}^{1-s_{r-1}}\alpha_r^{s_r}:(s_1,\dots,s_r)\in\{0,1\}^r\}$$
is a basis of integers in the field $\Q\tqrsqrt$ over $\Z$.
\end{enumerate}
\end{theorem}
\begin{proof}
Let $M=\Q\tqsqrt{r-1}$. Apply Lemma \ref{el.unramify} with $p=2$ and Theorem \ref{el.fundbasisstart}.
The ideal $2\Z$ is unramified in $M$. As shown above, this implies that $\sqrt{\tilde q_r}\not\in M$
and $(1,\sqrt{\tilde q_r})$ is a $M$-basis of $M(\sqrt{\tilde q_r})$.

Let $a+b\alpha_r$ be an algebraic integer and $a,b\in M$. The number $a-b\alpha_r$ is conjugate to
$a+b\alpha_r$ and therefore is also an algebraic integer. Thus, their sum $2a$ and product $a^2\mp2b^2$
are algebraic integers and lie in $\ro_M$. Furthermore, $(2a)^2-4(a^2\mp2b^2)=\pm2(2b)^2\in\ro_M$,
with Lemma \ref{el.ramint} this implies $2b\in\ro_M$. Now $2(a^2\mp2b^2)\pm(2b)^2=2a^2\in\ro_M$,
with Lemma \ref{el.ramint} this implies $a\in\ro_M$. Finally, $a^2-(a^2\mp2b^2)=\pm2b^2\in\ro_M$,
with Lemma \ref{el.ramint} this implies $b\in\ro_M$. Therefore, $(1,\alpha_r)$ is a $\ro_M$-basis
of the ring of integers in $M(\sqrt{\tilde q_r})$. Use of Theorem \ref{el.fundbasisstart}
concludes the proof.
\end{proof}

\begin{theorem}\label{el.fundbasisend}
Let $\tilde q_1,\dots,\tilde q_{r-1}$ be the same as in Theorem \ref{el.fundbasisstart}
and $\tilde q_r=-4$. Let $\alpha_r=\sqrt{\frac{\tilde q_r}4}=i$. Then:
\begin{enumerate}
\item The set $\{\alpha_1^{s_1}\dots\alpha_r^{s_r}:(s_1,\dots,s_r)\in\{0,1\}^r\}$
is a basis of integers in the field $\Q\tqrsqrt$ over $\Z$.
\item The set $$\{\tilde\alpha_1^{s_1}\alpha_1^{1-s_1}\dots\tilde\alpha_{r-1}^{s_{r-1}}\alpha_{r-1}^{1-s_{r-1}}\alpha_r^{s_r}:(s_1,\dots,s_r)\in\{0,1\}^r\}$$
is a basis of integers in the field $\Q\tqrsqrt$ over $\Z$.
\end{enumerate}
\end{theorem}
\begin{proof} Let $M=\Q\tqsqrt{r-1}$. The identity $2=-i(1+i)^2$ shows that the ideal $2\Z$
is ramified in any field containing $i$. Lemma \ref{el.unramify} and Theorem \ref{el.fundbasisstart}
imply that $2\Z$ is unramified in $M$. Therefore, $i\not\in M$.

Let $a+bi$ be an algebraic integer and $a,b\in M$. The number $a-bi$ is conjugate to $a+bi$
and therefore is also an algebraic integer. Thus, their sum $2a$ and product $a^2+b^2$ are also
algebraic integers and lie in $\ro_M$. Furthermore, $2(a^2+b^2)+2a\cdot 2b=2(a+b)^2\in\ro_M$,
so Lemmas \ref{el.unramify} and \ref{el.ramint} with $p=2$ and Theorem \ref{el.fundbasisstart}
imply that $a+b\in\ro_M$. Now $2a-(a+b)=a-b\in\ro_M$, $(a+b)(a-b)=a^2-b^2\in\ro_M$,
$2a^2\in\ro_M$, $2b^2\in\ro_M$. Applying Lemmas \ref{el.unramify}, \ref{el.ramint} and Theorem
\ref{el.fundbasisstart} again, we obtain $a,b\in\ro_M$. Thus, $(1,\alpha_t)$ is a $\ro_M$-basis
of the ring of integers in $M(\sqrt{\tilde q_r})$. Use of Theorem \ref{el.fundbasisstart}
concludes the proof.
\end{proof}

Let $\oplus$ denote the addition of integer numbers modulo 2.

In each case we have $[K_G:\Q]=2^t$. Thus, $\sqrt{q_j^*}\not\in\Q(\dots,\sqrt{q_{j-1}^*},\sqrt{q_{j+1}^*},\dots)$
for any $1\le j\le t$. Therefore, $\Gal(K_G/\Q)$ has $t$ elements $\tau_j$ with the following action:
\begin{equation}\label{el.tausingledef}
\tau_j\left(\sqrt{q_j^*}\right)=-\sqrt{q_j^*},\quad\tau_j\left(\sqrt{q_i^*}\right)=\sqrt{q_i^*}\mbox{ for }i\ne j.
\end{equation}
Let
$$
\tau'_\mu=\tau_1^{\mu_1}\ldots\tau_t^{\mu_t}\in\Gal(K_G/\Q)
$$
for $\mu\in\{0,1\}^t$. Comparing the action of $\tau'_\mu$ on $\sqrt{q_i^*}$,
it is easy to see that $\tau_\mu'$ are pairwise different. We obtain $2^t=|\Gal(K_G/\Q)|$
different elements of $\Gal(K_G/\Q)$, so this group does not contain other elements.

The theorems above give a $\Z$-basis of $\ro_{K_G}$. We also need the intersection
$\ro_{K_G}\cap\R$ (obviously, it is the ring of integers in $K_G\cap\R$) and
the intersection $\ro_{K_G}\cap i\R$ (obviously, it is a $\Z$-module).
There is at least one negative $q_i^*$. Let $u$ be the number of positive $q_i^*$,
$0\le u<t$. We assume without loss of generality that $q_1^*>0$, \dots, $q_u^*>0$,
$q_{u+1}^*<0$, \dots, $q_t^*<0$.

The complex conjugation acts on $\sqrt{q_i^*}$ same as the composition
$\tau_{u+1}\dots\tau_t$. Since $K_G\cap\R$ is the fixed field of
the complex conjugation restricted to $K_G$,
the group $\Gal((K_G\cap\R)/\Q)$
is isomorphic to the factorgroup of $\Gal(K_G/\Q)$ by the subgroup
generated by the complex conjugation. We select an element with $\mu_t=0$
as a representative in each coset and obtain that $\Gal((K_G\cap\R)/\Q)$
consists of the automorphisms
\begin{equation}\label{el.taudef}
\tau_\lambda=\tau_{\lambda_1,\dots,\lambda_{t-1}}=\tau'_{\lambda_1,\dots,\lambda_{t-1},0}=\tau_1^{\lambda_1}\ldots\tau_{t-1}^{\lambda_{t-1}}
\end{equation}
for $\lambda\in\{0,1\}^{t-1}$, $\tau_\lambda$ are pairwise different for different $\lambda$.

Note that $\sqrt{d}$ has two possible values. Further we select the value that equals
the product $\sqrt{q_1^*}\dots\sqrt{q_t^*}$, where the values of individual square
roots are the same as in definition of $\alpha_i$ and $\tilde\alpha_i$.

\begin{theorem}\label{el.basiscapr}
Let $q_1^*,\dots,q_t^*$ be as in Lemma \ref{el.qdef}, odd and numbered so that
$q_i^*>0$ for $1\le i\le u$, $q_i^*<0$ for $u<i\le t$, where $0\le u\le t-1$.
Let $K_G=\Q\qtsqrt$. Let $\tau_\lambda$ be defined by \eqref{el.taudef}.
\begin{enumerate}
\item Define
\begin{multline*}
\beta_{s_1,\dots,s_{t-1}}=\beta_{s_1,\dots,s_{t-1}}(q_1^*,\dots,q_t^*)\\
=\left(\prod_{i=1}^u\tilde\alpha_i^{s_i}\alpha_i^{1-s_i}\right)
\left(
\left(\prod_{i=u+1}^{t-1}\tilde\alpha_i^{s_i}\alpha_i^{1-s_i}\right)\alpha_t+
\left(\prod_{i=u+1}^{t-1}\tilde\alpha_i^{1-s_i}\alpha_i^{s_i}\right)\tilde\alpha_t
\right).
\end{multline*}
The set $\{\beta_{s_1,\dots,s_{t-1}}:(s_1,\dots,s_{t-1})\in\{0,1\}^{t-1}\}$ is
a $\Z$-basis of the ring of integers in $K_G\cap\R$.
\item Define
\begin{multline*}
\beta^*_{s_1,\dots,s_{t-1}}=\beta^*_{s_1,\dots,s_{t-1}}(q_1^*,\dots,q_t^*)\\
=\left(\prod_{i=1}^u(-\tilde\alpha_i)^{s_i}\alpha_i^{1-s_i}\right)
\left(
\left(\prod_{i=u+1}^{t-1}(-\tilde\alpha_i)^{s_i}\alpha_i^{1-s_i}\right)\alpha_t-
\left(\prod_{i=u+1}^{t-1}\tilde\alpha_i^{1-s_i}(-\alpha_i)^{s_i}\right)\tilde\alpha_t
\right).
\end{multline*}
The set $\{\beta^*_{s_1,\dots,s_{t-1}}:(s_1,\dots,s_{t-1})\in\{0,1\}^{t-1}\}$ is
a $\Z$-basis of the $\Z$-module $\ro_{K_G}\cap i\R$.
\item For any $\eta,\nu\in\{0,1\}^{t-1}$
$$
\sum_{\mu\in\{0,1\}^{t-1}}(-1)^{\mu_1+\ldots+\mu_{t-1}}\tau_\mu\left(\beta_{\eta_1,\dots,\eta_{t-1}}\beta^*_{\nu_1,\dots,\nu_{t-1}}\right)
=\begin{cases}
\sqrt d,&\mbox{ if }\eta=\nu,\\
0,&\mbox{ otherwise.}\end{cases}
$$
\end{enumerate}
\end{theorem}
\begin{proof}
Let $\beta'_{s_1,\dots,s_t}$ be the element of the basis from second assertion of Theorem
\ref{el.fundbasisstart} corresponding to the set $(s_1,\dots,s_t)$.

A number from $\ro_{K_G}$ is in $K_G\cap\R$ if and only if it is invariant under the
complex conjugation. It is easy to see that the complex conjugation maps $\beta'_{s_1,\dots,s_t}$
to $\beta'_{s_1,\dots,s_u,1-s_{u+1},\dots,1-s_t}$. Thus, a $\Z$-linear combination of $\beta'_{s_1,\dots,s_t}$
is invariant if and only if coefficients of $\beta'_{s_1,\dots,s_t}$ and $\beta'_{s_1,\dots,s_u,1-s_{u+1},\dots,1-s_t}$
are equal for any set $(s_i)$. Now Theorem \ref{el.fundbasisstart} implies that
$\{\beta'_{s_1,\dots,s_{t-1},0}+\beta'_{s_1,\dots,s_u,1-s_{u+1},\dots,1-s_{t-1},1}\}$ is a
required basis. From the definition of $\beta'$ it is easy to see that this sum is equal to
$\beta_{s_1,\dots,s_{t-1}}$. This concludes the proof of the first assertion.

A number from $\ro_{K_G}$ is in $K_G\cap i\R$ if and only if it changes the sign under
the complex conjugation. Similarly to the first assertion, we obtain that
$\{\beta'_{s_1,\dots,s_{t-1},0}-\beta'_{s_1,\dots,s_u,1-s_{u+1},1-s_{t-1},1}\}$ is
a required basis. From the definition of $\beta'$ it is easy to see that this difference
is equal to $\pm\beta^*_{s_1,\dots,s_{t-1}}$. This concludes the proof of the second assertion.

The last assertion is checked by a direct calculation. It is easy to see that
\begin{multline*}
\tau_\mu\left(\beta_{\eta_1,\dots,\eta_{t-1}}\right)=
\left(\prod_{i=1}^u\tilde\alpha_i^{\mu_i\oplus\eta_i}\alpha_i^{1-(\mu_i\oplus\eta_i)}\right)\\
\times\left(
\left(\prod_{i=u+1}^{t-1}\tilde\alpha_i^{\mu_i\oplus\eta_i}\alpha_i^{1-(\mu_i\oplus\eta_i)}\right)\alpha_t+
\left(\prod_{i=u+1}^{t-1}\tilde\alpha_i^{1-(\mu_i\oplus\eta_i)}\alpha_i^{\mu_i\oplus\eta_i}\right)\tilde\alpha_t
\right),
\end{multline*}
\begin{multline*}
\tau_\mu\left(\beta^*_{\nu_1,\dots,\nu_{t-1}}\right)=
\left(\prod_{i=1}^u(-1)^{\nu_i}\tilde\alpha_i^{\mu_i\oplus\nu_i}\alpha_i^{1-(\mu_i\oplus\nu_i)}\right)\\
\times\left(
\left(\prod_{i=u+1}^{t-1}(-1)^{\nu_i}\tilde\alpha_i^{\mu_i\oplus\nu_i}\alpha_i^{1-(\mu_i\oplus\nu_i)}\right)\alpha_t-
\left(\prod_{i=u+1}^{t-1}(-1)^{\nu_i}\tilde\alpha_i^{1-(\mu_i\oplus\nu_i)}\alpha_i^{\mu_i\oplus\nu_i}\right)\tilde\alpha_t
\right).
\end{multline*}
Substitute these formulas to the product $\tau_\mu(\beta_{\eta_1,\dots,\eta_{t-1}})\tau_\mu(\beta^*_{\nu_1,\dots,\nu_{t-1}})$,
obtain the formula of the form $(a+b)(c-d)$. Expand it and obtain four operands $ac+bc-ad-bd$.
Let $\delta_{ij}$ be the Kronecker delta: $\delta_{ii}=1$, $\delta_{ij}=0$ if $i\ne j$.
Note that
\begin{multline*}
\sum_{\mu_i=0}^1(-1)^{\mu_i}(-1)^{\nu_i}\tilde\alpha_i^{(\mu_i\oplus\eta_i)+(\mu_i\oplus\nu_i)}\alpha_i^{1-(\mu_i\oplus\eta_i)+1-(\mu_i\oplus\nu_i)}\\
=(-1)^{\nu_i}\left(\tilde\alpha_i^{\eta_i+\nu_i}\alpha_i^{2-(\eta_i+\nu_i)}-\tilde\alpha_i^{2-(\eta_i+\nu_i)}\alpha_i^{\eta_i+\nu_i}\right)=
\delta_{\eta_i\nu_i}\left(\alpha_i^2-\tilde\alpha_i^2\right)\\
=\delta_{\eta_i\nu_i}\sqrt{q_i^*},
\end{multline*}
\begin{multline*}
\sum_{\mu_i=0}^1(-1)^{\mu_i}(-1)^{\nu_i}\tilde\alpha_i^{1-(\mu_i\oplus\eta_i)+(\mu_i\oplus\nu_i)}\alpha_i^{(\mu_i\oplus\eta_i)+1-(\mu_i\oplus\nu_i)}\\
=(-1)^{\nu_i}\left(\tilde\alpha_i^{1-\eta_i+\nu_i}\alpha_i^{1+\eta_i-\nu_i}-\tilde\alpha_i^{1+\eta_i-\nu_i}\alpha_i^{1-\eta_i+\nu_i}\right)=
\delta_{\eta_i+\nu_i,1}\left(\alpha_i^2-\tilde\alpha_i^2\right)\\
=\delta_{\eta_i+\nu_i,1}\sqrt{q_i^*}
\end{multline*}
and transposing of $\alpha_i$ with $\tilde\alpha_i$ gives two more products with values multiplied by $(-1)$.

Therefore,
\begin{multline*}
\sum_{\mu\in\{0,1\}^{t-1}}(-1)^{\mu_1+\ldots+\mu_{t-1}}\tau_\mu\left(\beta_{\eta_1,\dots,\eta_{t-1}}\beta^*_{\nu_1,\dots,\nu_{t-1}}\right)\\
=\left(\prod_{i=1}^u\delta_{\eta_i\nu_i}\sqrt{q_i^*}\right)
\Bigg(
\alpha_t^2\prod_{i=u+1}^{t-1}\delta_{\eta_i\nu_i}\sqrt{q_i^*}+
\tilde\alpha_t\alpha_t\prod_{i=u+1}^{t-1}\delta_{\eta_i+\nu_i,1}\sqrt{q_i^*}\\
-\alpha_t\tilde\alpha_t\prod_{i=u+1}^{t-1}\left(-\delta_{\eta_i+\nu_i,1}\sqrt{q_i^*}\right)
-\tilde\alpha_t^2\prod_{i=u+1}^{t-1}\left(-\delta_{\eta_i\nu_i}\sqrt{q_i^*}\right)
\Bigg).
\end{multline*}
The sign of the product $q_1^*\dots q_t^*$ is defined by the parity of the number of negative factors.
There are $t-u$ negative factors, so the inequality $q_1^*\dots q_t^*=d<0$ implies that
$t-u$ is odd and $\prod_{i=u+1}^{t-1}(-1)=(-1)^{t-u-1}=1$.
\begin{multline*}
\sum_{\mu\in\{0,1\}^{t-1}}(-1)^{\mu_1+\ldots+\mu_{t-1}}\tau_\mu\left(\beta_{\eta_1,\dots,\eta_{t-1}}\beta^*_{\nu_1,\dots,\nu_{t-1}}\right)\\
=\left(\prod_{i=1}^u\delta_{\eta_i\nu_i}\sqrt{q_i^*}\right)
\left(\alpha_t^2-\tilde\alpha_t^2\right)
\left(\prod_{i=u+1}^{t-1}\delta_{\eta_i\nu_i}\sqrt{q_i^*}\right)
=\sqrt{q_t^*}\prod_{i=1}^{t-1}\delta_{\eta_i\nu_i}\sqrt{q_i^*}.
\end{multline*}
\end{proof}

\begin{theorem}
Let $q_2^*,\dots,q_t^*$ be the same as in Theorem \ref{el.basiscapr}, and $q_1^*=8$.
Let $q_i^*$ be numbered so that $q_i^*>0$ for $1\le i\le u$ and $q_i^*<0$ for $u<i\le t$,
where $1\le u\le t-1$. Let $K_G=\Q\qtsqrt$. Let $\tau_\lambda$ be defined by \eqref{el.taudef}.
\begin{enumerate}
\item Define
\begin{multline*}
\beta_{s_1,\dots,s_{t-1}}=\beta_{s_1,\dots,s_{t-1}}(q_1^*,\dots,q_t^*)=\sqrt{2}^{s_1}
\left(\prod_{i=2}^u\tilde\alpha_i^{s_i}\alpha_i^{1-s_i}\right)\\
\times\left(
\left(\prod_{i=u+1}^{t-1}\tilde\alpha_i^{s_i}\alpha_i^{1-s_i}\right)\alpha_t+
\left(\prod_{i=u+1}^{t-1}\tilde\alpha_i^{1-s_i}\alpha_i^{s_i}\right)\tilde\alpha_t
\right).
\end{multline*}
The set $\{\beta_{s_1,\dots,s_{t-1}}:(s_1,\dots,s_{t-1})\in\{0,1\}^{t-1}\}$ is
a $\Z$-basis of the ring of integers in $K_G\cap\R$.
\item Define
\begin{multline*}
\beta^*_{s_1,\dots,s_{t-1}}=\beta^*_{s_1,\dots,s_{t-1}}(q_1^*,\dots,q_t^*)=\sqrt{2}^{1-s_1}
\left(\prod_{i=2}^u(-\tilde\alpha_i)^{s_i}\alpha_i^{1-s_i}\right)\\
\times\left(
\left(\prod_{i=u+1}^{t-1}(-\tilde\alpha_i)^{s_i}\alpha_i^{1-s_i}\right)\alpha_t-
\left(\prod_{i=u+1}^{t-1}\tilde\alpha_i^{1-s_i}(-\alpha_i)^{s_i}\right)\tilde\alpha_t
\right).
\end{multline*}
The set $\{\beta^*_{s_1,\dots,s_{t-1}}:(s_1,\dots,s_{t-1})\in\{0,1\}^{t-1}\}$ is
a $\Z$-basis of the $\Z$-module $\ro_{K_G}\cap i\R$.
\item For any $\eta,\nu\in\{0,1\}^{t-1}$
$$
\sum_{\mu\in\{0,1\}^{t-1}}(-1)^{\mu_1+\ldots+\mu_{t-1}}\tau_\mu\left(\beta_{\eta_1,\dots,\eta_{t-1}}\beta^*_{\nu_1,\dots,\nu_{t-1}}\right)
=\begin{cases}
\sqrt d,&\mbox{ if }\eta=\nu,\\
0,&\mbox{ otherwise.}\end{cases}
$$
\end{enumerate}
\end{theorem}
\begin{proof}
The arguments are similar to Theorem \ref{el.basiscapr}. Calculating the
expression from the third assertion yields an additional factor
$$
\sum_{\mu_1=0}^1(-1)^{\mu_1}\left((-1)^{\mu_1}\sqrt 2\right)^{\eta_1}
\left((-1)^{\mu_1}\sqrt 2\right)^{1-\nu_1}
=\sqrt{2}^{1+\eta_1-\nu_1}\left(1+(-1)^{\eta_1+\nu_1}\right)
=2\sqrt2\delta_{\eta_1\nu_1}.
$$
\end{proof}

\begin{theorem}
Let $q_1^*,\dots,q_{t-1}^*$ be the same as in Theorem \ref{el.basiscapr}, and
$q_t^*\in\{-4,-8\}$. Let $q_i^*$ be numbered so that $q_i^*>0$ for $1\le i\le u$
and $q_i^*<0$ for $u<i\le t$, where $0\le u\le t-2$. Let $K_G=\Q\qtsqrt$.
Let $\tau_\lambda$ be defined by \eqref{el.taudef}.
\begin{enumerate}
\item Define
\begin{multline*}
\beta_{s_1,\dots,s_{t-1}}=\beta_{s_1,\dots,s_{t-1}}(\sqrt{q_1^*},\dots,\sqrt{q_t^*})=
\left(\prod_{i=1}^u\tilde\alpha_i^{s_i}\alpha_i^{1-s_i}\right)\\ \times
\Bigg(
\left(\prod_{i=u+1}^{t-2}\tilde\alpha_i^{s_i}\alpha_i^{1-s_i}\right)\alpha_{t-1}\alpha_t^{s_{t-1}}
+\left(\prod_{i=u+1}^{t-2}\tilde\alpha_i^{1-s_i}\alpha_i^{s_i}\right)\tilde\alpha_{t-1}(-\alpha_t)^{s_{t-1}}
\Bigg).
\end{multline*}
The set $\{\beta_{s_1,\dots,s_{t-1}}:(s_1,\dots,s_{t-1})\in\{0,1\}^{t-1}\}$ is
a $\Z$-basis of the ring of integers in $K_G\cap\R$.
\item Define
\begin{multline*}
\beta^*_{s_1,\dots,s_{t-1}}=\beta^*_{s_1,\dots,s_{t-1}}(\sqrt{q_1^*},\dots,\sqrt{q_t^*})=
\left(\prod_{i=1}^u(-\tilde\alpha_i)^{s_i}\alpha_i^{1-s_i}\right)\\ \times
\Bigg(
\left(\prod_{i=u+1}^{t-2}(-\tilde\alpha_i)^{s_i}\alpha_i^{1-s_i}\right)\alpha_{t-1}\alpha_t^{1-s_{t-1}}
-\left(\prod_{i=u+1}^{t-2}\tilde\alpha_i^{1-s_i}(-\alpha_i)^{s_i}\right)\tilde\alpha_{t-1}(-\alpha_t)^{1-s_{t-1}}
\Bigg).
\end{multline*}
The set $\{\beta^*_{s_1,\dots,s_{t-1}}:(s_1,\dots,s_{t-1})\in\{0,1\}^{t-1}\}$ is
a $\Z$-basis of the $\Z$-module $\ro_{K_G}\cap i\R$.
\item For any $\eta,\nu\in\{0,1\}^{t-1}$
$$
\sum_{\mu\in\{0,1\}^{t-1}}(-1)^{\mu_1+\ldots+\mu_{t-1}}\tau_\mu\left(\beta_{\eta_1,\dots,\eta_{t-1}}\beta^*_{\nu_1,\dots,\nu_{t-1}}\right)
=\begin{cases}
\sqrt d,&\mbox{ if }\eta=\nu,\\
0,&\mbox{ otherwise.}\end{cases}
$$
\end{enumerate}
\end{theorem}
\begin{proof}
Let $\beta'_{s_1,\dots,s_t}$ be the element of the basis from second assertion of Theorem
\ref{el.fundbasisend} corresponding to the set $(s_1,\dots,s_t)$.

A number from $\ro_{K_G}$ is in $K_G\cap\R$ if and only if it is invariant under the
complex conjugation. It is easy to see that the complex conjugation maps $\beta'_{s_1,\dots,s_t}$
to $(-1)^{s_t}\beta'_{s_1,\dots,s_u,1-s_{u+1},\dots,1-s_{t-1},s_t}$.
Thus, a $\Z$-linear combination of $\beta'_{s_1,\dots,s_t}$
is invariant if and only if coefficients of $\beta'_{s_1,\dots,s_t}$ and $\beta'_{s_1,\dots,s_u,1-s_{u+1},\dots,1-s_{t-1},s_t}$
are the same for $s_t=0$ and differ in the sign for $s_t=1$. Now Theorem \ref{el.fundbasisend} implies that
$\{\beta'_{s_1,\dots,s_{t-2},0,s_t}+(-1)^{s_t}\beta'_{s_1,\dots,s_u,1-s_{u+1},\dots,1-s_{t-2},1,s_t}\}$ is a
required basis. From the definition of $\beta'$ it is easy to see that this sum is equal to
$\beta_{s_1,\dots,s_{t-2},s_t}$. This concludes the proof of the first assertion.

The second assertion is proved similarly to the first one.

The third assertion is checked by a direct calculation. Similar to the proof of Theorem \ref{el.basiscapr}
we obtain
\begin{multline*}
\sum_{\mu\in\{0,1\}^{t-1}}(-1)^{\mu_1+\ldots+\mu_{t-1}}\tau_\mu\left(\beta_{\eta_1,\dots,\eta_{t-1}}\beta^*_{\nu_1,\dots,\nu_{t-1}}\right)\\
=\left(\prod_{i=1}^u\delta_{\eta_i\nu_i}\sqrt{q_i^*}\right)
\alpha_t^{\eta_{t-1}+1-\nu_{t-1}}\Bigg(
\sqrt{q_{t-1}^*}\prod_{i=u+1}^{t-2}\delta_{\eta_i\nu_i}\sqrt{q_i^*}\\
+(-1)^{\nu_{t-1}+\eta_{t-1}}(-\sqrt{q_{t-1}^*})\prod_{i=u+1}^{t-2}\left(-\delta_{\eta_i\nu_i}\sqrt{q_i^*}\right)
\Bigg).
\end{multline*}
Since $q_1^*\dots q_t^*<0$, the number of negative $q_i^*$ (i.e. $t-u$) is odd. Therefore,
$\prod_{i=u+1}^{t-2}(-1)=(-1)^{t-u-2}=-1$.
\begin{multline*}
\sum_{\mu\in\{0,1\}^{t-1}}(-1)^{\mu_1+\ldots+\mu_{t-1}}\tau_\mu\left(\beta_{\eta_1,\dots,\eta_{t-1}}\beta^*_{\nu_1,\dots,\nu_{t-1}}\right)\\
=\left(\prod_{i=1}^{t-2}\delta_{\eta_i\nu_i}\sqrt{q_i^*}\right)
(1+(-1)^{\nu_{t-1}+\eta_{t-1}})\alpha_t^{1+\eta_{t-1}-\nu_{t-1}}\\
=\delta_{\eta_{t-1}\nu_{t-1}}\left(\prod_{i=1}^{t-2}\delta_{\eta_i\nu_i}\sqrt{q_i^*}\right)2\alpha_t.
\end{multline*}
\end{proof}

\begin{theorem}\label{el.basiscaprend}
Let $q_1^*,\dots,q_{t-1}^*$ are positive odd, $q_t^*=-4$ or $q_t^*=-8$.
\begin{enumerate}
\item Define
$$
\beta_{s_1,\dots,s_{t-1}}=\beta_{s_1,\dots,s_{t-1}}(q_1^*,\dots,q_t^*)
=\prod_{i=1}^{t-1}\tilde\alpha_i^{s_i}\alpha_i^{1-s_i}.
$$
The set $\{\beta_{s_1,\dots,s_{t-1}}:(s_1,\dots,s_{t-1})\in\{0,1\}^{t-1}\}$ is
a $\Z$-basis of the ring of integers in $K_G\cap\R$.
\item Define
$$
\beta^*_{s_1,\dots,s_{t-1}}=\beta^*_{s_1,\dots,s_{t-1}}(q_1^*,\dots,q_t^*)
=\left(\prod_{i=1}^{t-1}(-\tilde\alpha_i)^{s_i}\alpha_i^{1-s_i}\right)\sqrt{q_t^*}.
$$
The set $\{\beta^*_{s_1,\dots,s_{t-1}}:(s_1,\dots,s_{t-1})\in\{0,1\}^{t-1}\}$ is
a $\Z$-basis of the $\Z$-module $\ro_{K_G}\cap i\R$.
\item For any $\eta,\nu\in\{0,1\}^{t-1}$
$$
\sum_{\mu\in\{0,1\}^{t-1}}(-1)^{\mu_1+\ldots+\mu_{t-1}}\tau_\mu\left(\beta_{\eta_1,\dots,\eta_{t-1}}\beta^*_{\nu_1,\dots,\nu_{t-1}}\right)
=\begin{cases}
\sqrt d,&\mbox{ if }\eta=\nu,\\
0,&\mbox{ otherwise.}\end{cases}
$$
\end{enumerate}
\end{theorem}
\begin{proof}
Obviously, here $K_G=M(\sqrt{q_t^*})$ with $M\subset\R$. Thus $K_G\cap\R=M$, $K_G\cap i\R=\sqrt{q_t^*}\cdot M$.
First two assertions follow from Theorem \ref{el.fundbasisstart}. The last assertion is checked by
a direct calculation similar to the one from the proof of Theorem \ref{el.basiscapr}.
\end{proof}

For convenience, we denote $\beta_\mu=\beta_{\mu_1,\dots,\mu_{t-1}}$ for $\mu\in\{0,1\}^{t-1}$.
Let $\mathcal M$ denote the field $K_G\cap\R$.
The set $\{\beta_\mu\}$ is a $\Z$-basis of $\ro_{\mathcal M}$.

Let $z$ be any element of $\ro_{K_G}$. Since $z\in K_G$, also $\overline z\in K_G$ and
$z+\overline z=2\re z\in K_G\cap\R$ and $z-\overline z=2i\im z\in K_G\cap i\R$. Moreover,
$z$ and $\overline z$ are algebraic integers, so $2\re z$ and $2\im z$ are algebraic
integers too. Thus, $2\re z=\sum_\mu b_\mu\beta_\mu$ and $2i\im z=\sum_\mu b'_\mu\beta^*_\mu$.
Hereafter sums with parameter given by a greek letter without an explicit range is assumed
to be over $\{0,1\}^{t-1}$. We want to find the integer numbers $b_\mu$ and $b'_\mu$ by
an approximate values of these sums. The numbers $\beta_\mu$ form a basis of a real field
$\mathcal M$. The basis $\beta'_\mu$ is pure imaginary and becomes a basis of the same field
$\mathcal M$ after dividing e.g. by $\sqrt{q_t^*}\in K_G$, $q_t^*<0$. Thus, to find an
exact expression for $z$ by an approximate value, it is sufficient to solve the next task:
restore the coefficients of the decomposition of a number given by a sufficiently accurate
approximation, by a real basis.

The scheme of next sections is following.
\begin{itemize}
\item Consider a divisor of the polynomial $H_D[\theta,\alpha_*]$ over the field $K_G$.
The degree of this divisor is $\frac h{2^{t-1}}$. Section \ref{divider} deals with
this task. The ultimate goal is to use
this divisor instead of the full polynomial, thus decreasing the number and
the magnitude of coefficients to be calculated.
\item Calculate an apriori upper bound for all conjugates to coefficients of the divisor.
This is done in Section \ref{el.upbound}.
\item The main idea for calculating exact values is to use simultaneous rational approximations
to the elements of a basis. Section \ref{el.simapprox} shows how to construct such
approximations for $\beta_\mu$ and $\beta^*_\mu$ with any predefined precision. The actual
precision depends on the bound from Section \ref{el.upbound}.
\item Finally, Section \ref{el.get} shows how to calculate exact values by
approximations. Also Section \ref{el.get} sums up all the steps used in
our optimization.
\end{itemize}

\section{Divisor of $H_D[\theta,\alpha_*](x)$}\label{divider}
Let $\tilde{\mathfrak a}\in\mathcal H_D$. Select a form $(A,B,C)$ such that
$\mathfrak h(A,B,C)=\tilde{\mathfrak a}$ and $\gcd(A,D)=1$; this is possible
because $\mathfrak h$ depends only on the equivalence class of a form and
each class contains a form $(A,B,C)$ with $\gcd(A,D)=1$ due to \cite[Lemmas 2.25 and 2.3]{cox}.
Let $\varphi:\mathcal H_D\to\{\pm1\}^t$ be the map defined by the formula
$$
\varphi(\tilde{\mathfrak a})=\Bigg(\left(\frac{q_1^*}{A}\right),\dots,\left(\frac{q_t^*}{A}\right)\Bigg).
$$
This definition is correct because the Artin map depends only on an ideal class in
$I(\ro,f)/P_{K,\Z}(f)$ and Theorem \ref{el.artinaction} implies that $\left(\frac{q_i^*}A\right)$
does not change when a form $(A,B,C)$ is replaced to an equivalent form.

\begin{theorem}\label{el.phiim}
The image of the map $\varphi$ is the group $\{(\varepsilon_1,\dots,\varepsilon_t)\in\{\pm1\}^t:
\prod_i\varepsilon_i=1\}$. The map $\varphi$ is a group homomorphism. The fixed field
$L^{\Omega(\Ker\varphi)}=\{x\in L:\tau(x)=x\mbox{ for all }\tau\in\Omega(\Ker\varphi)\}$
is $K\qtsqrt$.
\end{theorem}
\begin{proof}
The assertion 3 of Theorem \ref{el.artinaction} and the fact that the Artin map
is a homomorphism imply that $\varphi$ is a homomorphism.

Let $\mathfrak a$ be the ideal for the form $(A,B,C)$ defined in Theorem \ref{el.artinaction}.
We have
$$
\left(\frac{q_i^*}{A}\right)=\frac1{\sqrt{q_i^*}}\left(\frac{L/K}{\Omega_1(\mathfrak a)}\right)(\sqrt{q_i^*}).
$$
Multiplying over all $i$ and using Lemma \ref{el.qdef}, we obtain
$$
\left(\frac{q_1^*}{A}\right)\ldots\left(\frac{q_t^*}{A}\right)=\frac1{\sqrt d}
\left(\frac{L/K}{\Omega_1(\mathfrak a)}\right)(\sqrt d).
$$
Since $\sqrt d\in K$ and $\left(\frac{L/K}{\Omega_1(\mathfrak a)}\right)$ is an element
of $\Gal(L/K)$, the right-hand side equals $1$. This proves the inclusion
of image of $\varphi$ to $\{(\varepsilon_i)\in\{\pm1\}^t:\prod_i\varepsilon_i=1\}$.

Let $\tilde{\mathfrak a}$ lie in the kernel of $\varphi$ (i.e. $\varphi(\tilde{\mathfrak a})=(1,\dots,1)$).
Let $\mathfrak a$ be the representative of $\tilde{\mathfrak a}$ from the assertion 2 of Theorem \ref{el.artinaction}.
Then
$$
\left(\frac{L/K}{\Omega_1(\mathfrak a)}\right)(\sqrt{q_i^*})=\sqrt{q_i^*}.
$$
Equivalently, the image of $\Omega_1(\mathfrak a)$ under the Artin map acts trivially
on all $\sqrt{q_i^*}$. Due to the commutativity of the diagram \eqref{el.omegadef}
this image equals $\Omega(\tilde{\mathfrak a})$. This proves the inclusion
$K\qtsqrt\subset L^{\Omega(\Ker\varphi)}$.

According to Galois theory, $\Gal(L^{\Omega(\Ker\varphi)}/K)\cong\Gal(L/K)/\Omega(\Ker\varphi)
\cong\mathcal H_D/\Ker\varphi\cong\im\varphi$. In particular, $[L^{\Omega(\Ker\varphi)}:K]=|\im\varphi|\le2^{t-1}$.
We proved in Section \ref{el.integers} that $[K\qtsqrt:\Q]=2^t$, so $[K\qtsqrt:K]=2^{t-1}$.
Thus, the chain of inequalities $[K\qtsqrt:K]\le[L^{\Omega(\Ker\varphi)}:K]=|\im\varphi|\le2^{t-1}$
is possible only if $|\im\varphi|=2^{t-1}$ and $K\qtsqrt=L^{\Omega(\Ker\varphi)}$.
\end{proof}

We suggest to calculate the polynomial
\begin{equation}\label{el.hdused}
\hat H_D[\theta,\alpha_*](x)=\prod_{i:\varphi(\mathfrak h(A_i,B_i,C_i))=(1,\dots,1)}(x-\theta(\alpha_i)),
\end{equation}
which obviously divides $H_D[\theta,\alpha_*]$, instead of the entire polynomial $H_D[\theta,\alpha_*]$.
Here the function $\theta$ and the $N$-system $\{(A_i,B_i,C_i)\}$ satisfy the assumption of one of
Theorems \ref{el.t1}--\ref{el.t3}, and $\alpha_i$ is the root of the form $(A_i,B_i,C_i)$.

The main obstacle is that $\hat H_D[\theta,\alpha_*]$ is not invariant under $\Gal(L/K)$
and therefore does not lie in $\Q[x]$. Note that $\varphi$ is a homomorphism. Using
the formula \eqref{el.galaction2}, it is easy to see that $\Omega(\Ker\varphi)$
fixes $\hat H_D[\theta,\alpha_*](x)$, therefore, this polynomial has coefficients in $K_G$.
All numbers $\theta(\alpha)$ are algebraic integers (Theorems \ref{el.t1}--\ref{el.t3}),
so the coefficients of $\hat H_D[\theta,\alpha_*]$ are also algebraic integers. Therefore,
to use the polynomial $\hat H_D[\theta,\alpha_*]$ in the complex multiplication method,
one must know how to recover an algebraic integer from $K_G$ by its sufficiently accurate
approximation. Assuming that such a procedure is implemented, the other actions to
generate an elliptic curve are the same as in the original method.

An idea to use the genus field in the CM method was already considered in \cite{atkinmorain} (1993).
There the main obstacle for an algebraic integer $z$ is solved in the following way.
All conjugates of $z$ are calculated. One looks for the exact value of $z$ in the
form of linear combination of some generators with unknown coefficients. Any conjugate of $z$
is a linear combination of conjugates to generators with the same unknown coefficients.
The known approximations for all conjugates give a system of linear equations for these coefficients,
it allows to calculate them (approximately and then round to integer). We refer to
\cite{atkinmorain} for the details. Note that this solution requires to calculate
values $\theta(\alpha_i)$ for roots of all elements of a $N$-system and
all conjugate polynomials to $\hat H_D[\theta,\alpha_*]$. Thus the optimization is only
in the magnitude of the coefficients.

Our approach requires to calculate only the polynomial $\hat H_D[\theta,\alpha_*]$
itself (although with a greater precision); in particular, it is sufficient to
know only values $\theta(\alpha_i)$ for roots of forms $\mathfrak a$ with
$\varphi(\mathfrak h(\mathfrak a))=(1,\dots,1)$. Theorem \ref{el.phiim}
obviously implies that the number of these forms is $2^{t-1}$ times less than
size of the $N$-system.

\section{Bound for coefficients of $\hat H_D[\theta,\alpha_*]$}\label{el.upbound}
According to Theorems \ref{el.basiscapr}--\ref{el.basiscaprend}, each coefficient
of the polynomial $\hat H_D[\theta,\alpha_*]$ can be represented with
a formula
$\frac12\left(\sum_\mu b_\mu\beta_\mu+\sum_\mu b'_\mu\beta^*_\mu\right)$,
where $b_\mu,b'_\mu\in\Z$, $\beta_\mu\in\R$, $\beta^*_\mu\in i\R$.
We need a bound for all conjugates,
$$
\left|\frac12\tau'_\lambda\left(\sum_\mu b_\mu\beta_\mu+\sum_\mu b'_\mu\beta^*_\mu\right)\right|\le T_0.
$$

Note that the polynomial $\hat H_D[j,\alpha_*]$ does not depend on the set $\alpha_*$,
so the short notation $\hat H_D[j]=\hat H_D[j,\alpha_*]$ is correct.

For theoretical bounds we apply the method from \cite{perfbase}.

Let us consider along with $\hat H_D[j]$ also polynomials
\begin{equation}\label{el.hdef}
\hat H_{D,\varphi_0}[j](x)=\prod_{i:\varphi(\mathfrak h(A_i,B_i,C_i))=\varphi_0}(x-j(\alpha_i)),
\end{equation}
where $\varphi_0\in\{0,1\}^t$, $(A_i,B_i,C_i)$ runs over representatives of all form classes,
$\alpha_i$ is the root of $(A_i,B_i,C_i)$.

By definition, $\hat H_D[j]=\hat H_{D,(1,\dots,1)}[j]$. Similarly to $\hat H_D[j]$,
the polynomial $\hat H_{D,\varphi_0}[j]$ is in $\ro_{K_G}[x]$ for each $\varphi_0$.
Moreover, if $\sigma\in\Gal(L/\Q)$ is the automorphism corresponding to an ideal class
$\mathfrak b\in\mathcal H_D$, then $\hat H_{D,\varphi_0}[j]^\sigma=\hat H_{D,\varphi_0\varphi(\mathfrak b)^{-1}}[j]$
due to Corollary \ref{el.galaction.conv}. Since any automorphism of the field $K_G$ can be extended
to an element of $\Gal(L/\Q)$, for each $\tau\in\Gal(K_G/\Q)$ there exists $\varphi_0=\varphi_0(\tau)$
such that $\hat H_{D,\varphi_1}[j]^\tau=\hat H_{D,\varphi_1\varphi_0(\tau)}[j]$ for any $\varphi_1\in\{\pm1\}^t$.

\begin{theorem}
The absolute value of each coefficient of the polynomial $\hat H_{D,\varphi_0}[j]$ does not exceed
$$
\exp\left(c_5h+c_1N\left(\ln^2N+4\gamma\ln N+c_6+\frac{\ln N+\gamma+1}N\right)\right)
$$
$$
\le\exp\left(c_1N\ln^2N+c_2N\ln N+c_3N+c_1\ln N+c_4\right)=T_0,
$$
where $N=\sqrt{\frac{|D|}3}$, $\gamma=0.577...$ is the Euler constant,
$c_1=\sqrt3\pi=5.441...$, $c_2=18.587...$, $c_3=17.442...$, $c_4=11.594...$,
$c_5=3.011...$, $c_6=2.566...$ The asymptotic upper bound
$$
T_0=\exp O\left(\sqrt{|D|}\ln^2|D|\right)
$$
holds for other functions $\theta$ too.
\end{theorem}
\begin{proof}
We follow \cite[Section 4]{perfbase}.

We can assume that $(A_i,B_i,C_i)$ in \eqref{el.hdef} are reduced forms, because
a change of a form to an equivalent form corresponds to some $SL_2(\Z)$-transformation
of the form root and the function $j$ is invariant under these. Let $(A,B,C)$ be
a reduced form; we need an upper bound for $\left|j\left(\frac{-B+\sqrt D}{2A}\right)\right|$.
The argument of $j$ lies in the area $\{z\in\h:|z|\ge1, |\re z|\le\frac12\}$. Therefore,
$\im z\ge\frac{\sqrt3}2$ and $|q|=|e^{2\pi iz}|\le e^{-\pi\sqrt{3}}$. Furthermore,
$$
j(z)=\frac1q+744+\sum_{m=1}^\infty c_mq^m,
$$
where $|c_m|\le\frac{e^{4\pi\sqrt m}}{\sqrt2m^{3/4}}$ due to \cite{jbound}. Thus,
$$
\left|j\left(\frac{-B+\sqrt D}{2A}\right)-\frac1q\right|\le744+\sum_{m=1}^\infty\frac{e^{4\pi\sqrt m}}{\sqrt 2m^{3/4}}e^{-\pi\sqrt3m}
=k_1=2114.566...
$$
and $\left|j\left(\frac{-B+\sqrt D}{2A}\right)\right|\le\frac1{|q|}+k_1\le\frac{k_2}{|q|}$
with $k_2=1+k_1e^{-\pi\sqrt3}=10.163...$.

Assume that all reduced forms are numbered so that $\{(A_i,B_i,C_i):1\le i\le\deg\hat H_D[j]\}$
are all reduced forms from the product \eqref{el.hdef} ordered by increasing
$\left|\frac1{q_i}\right|=e^{\pi\sqrt{|D|}/A_i}$. The absolute value of
the coefficient of $x^k$ in $\hat H_{D,\varphi_0}[j]$ does not exceed
$$
C_{\deg\hat H_D[j]}^k\prod_{i=k+1}^{\deg\hat H_D[j]}\frac{k_2}{|q_i|}
\le(2k_2)^{h/2^{t-1}}\prod_{i=1}^{h/2^{t-1}}e^{\pi\sqrt{|D|}/A_i}.
$$
Therefore, the logarithm of any coefficient of $\hat H_{D,\varphi_0}[j]$ does not exceed
$$
\frac h{2^{t-1}}\ln(2k_2)+\pi\sqrt{|D|}\sum_{i=1}^{h/2^{t-1}}\frac1{A_i}
\le h\ln(2k_2)+\pi\sqrt{|D|}\sum_{i=1}^h\frac1{A_i}.
$$
The bound for the last sum proved in \cite[Theorem 1.2]{perfbase} concludes the proof for $j$.

The bound for other functions $\theta$ follows from the proved one and \cite[Proposition 3]{invariants}.
\end{proof}

In practice it is better to use heuristic, but more accurate bounds.

The article \cite{invariants} suggests the following upper bound for logarithms of absolute values
of coefficients of the polynomial $H_D[j]$:
$$
\pi\sqrt{|D|}\sum_{(A,B,C)}\frac1A,
$$
where the sum is over all reduced forms. This bound is heuristic, but sufficiently close to the exact value.
The same article suggests multiplying this sum by some constant depending on $\theta$ to obtain
the analogous bound for $H_D[\theta]$. The constant is the ratio $\frac{\deg_j\Phi}{\deg_\theta\Phi}$,
where a polynomial $\Phi$ in two variables links functions $\theta$ and $j$ so that $\Phi(\theta(z),j(z))=0$.

Trivial changes of the arguments from \cite{invariants} with respect to $\hat H_D[j]$ give the
heuristic bound
\begin{equation}\label{el.t0used}
\ln T_0\sim\pi\sqrt{|D|}\max_{\varepsilon\in\{\pm1\}^t}\sum_{\substack
{(A,B,C):\varphi(\mathfrak h(A,B,C))=\varepsilon}}\frac1A
\end{equation}
for the invariant $j$. Again, for other invariants this bound should be multiplied by
$\frac{\deg_j\Phi}{\deg_\theta\Phi}$.

Let
$$
z=\frac12\left(\sum_\mu b_\mu\beta_\mu+\sum_\mu b'_\mu\beta^*_\mu\right)
\mbox{ be a coefficient of the polynomial }\hat H_{D,\varphi_0}[j],
$$
$b_\mu,b'_\mu\in\Z$. As mentioned above, the action of $\Gal(K_G/\Q)$ maps
the polynomial $\hat H_{D,\varphi_1}[j]$ to the polynomial of the same type.
So for any $\lambda\in\{0,1\}^t$ the following inequation holds:
$$
|\tau'_\lambda(z)|\le T_0.
$$

\section{Construction of rational approximations to a basis of ring of algebraic integers}
\label{el.simapprox}
There is a number of different algorithms for constructing simultaneous rational approximations
to a given set of real numbers. The book \cite{brentjes} covers many of them.
Properties of approximations differ significantly for different algorithms.
For practical purposes the inner product algorithm from \cite[Chapter 6A]{brentjes}
seems to be the best in the general case. Unfortunately, it is quite difficult to
prove good theoretical bounds for universal algorithms. Therefore we suggest
another algorithm which allows to obtain theoretical bounds, but
works only for very specific sets.

In essence, the main part of the following theorem is contained in the article \cite{peck}.
Main differences between the following theorem and \cite{peck} are
following: the explicit formulation, including explicit constants; the function
$\mathfrak M$ (\cite{peck} deals with dual basises which is equivalent to $\mathfrak M=1$);
specialization for our case (\cite{peck} does not require for $M/\Q$ to be Galois
and also contains a converse theorem).

\begin{theorem}\label{el.finalapprox}
Let $M\subset\R$ be a field such that $M/\Q$ is a Galois extension of degree $m$.
Let $W_1,\dots,W_m$ and $W_1^*,\dots,W_m^*$ be two basises of $M$.
Let $\mathfrak M:\Gal(M/\Q)\to\R$ be a function (not necessarily a homomorphism)
such that for each $1\le l,l'\le m$ the following equality holds:
$$
\sum_{\tau\in\Gal(M/\Q)}\mathfrak M(\tau)\tau(W_l W_{l'}^*)=
\begin{cases}1,&\mbox{if }l=l',\\ 0,&\mbox{if }l\ne l'.\end{cases}
$$
Let
$$
C=\sum_{\substack{\tau\in\Gal(M/\Q)\\ \tau\ne Id}}\left|\mathfrak M(\tau)\tau(W_1)\right|
$$
and
$$
C_i=\sum_{\substack{\tau\in\Gal(M/\Q)\\ \tau\ne Id}}\left|\mathfrak M(\tau)\left(\tau(W_i)-W_i\frac{\tau(W_1)}{W_1}\right)\right|
$$
for $i=2,\dots,m$. Let a positive number $\Delta$ and integers $\Lambda_1,\dots,\Lambda_m$ satisfy
the inequalities
$$
\sum_{i=1}^m \Lambda_iW_i^*=Z\ge1,
$$
$$
\left|\tau\left(\sum_{i=1}^m \Lambda_iW_i^*\right)\right|\le\frac{\Delta}{Z^{\frac1{m-1}}}\quad\mbox{ for each }\tau\in\Gal(M/\Q),\tau\ne Id.
$$
Then:
\begin{itemize}
\item $|\Lambda_1|\ge|\mathfrak M(Id)W_1|Z-C\Delta$.
\item If $|\Lambda_1|>C\Delta$, then $\mathfrak M(Id)\ne0$ and the following bound holds for each $i=2,\dots,m$:
$$
\left|\frac{\Lambda_i}{\Lambda_1}-\frac{W_i}{W_1}\right|\le C_i
\frac{\Delta}{|\Lambda_1|\left(\frac{|\Lambda_1|-C\Delta}{|\mathfrak M(Id)W_1|}\right)^{\frac1{m-1}}}.
$$
\end{itemize}
\end{theorem}
\begin{proof}
For each $l=1,\dots,m$
\begin{multline}\label{el.finalapprox.main}
\Lambda_l=\sum_{l'=1}^m\Lambda_{l'}\left(\sum_{\tau\in\Gal(M/\Q)}\mathfrak M(\tau)\tau\left(W_lW^*_{l'}\right)\right)\\
=\sum_{\tau\in\Gal(M/\Q)}\mathfrak M(\tau)\tau(W_l)\left(\sum_{l'=1}^m\Lambda_{l'}\tau(W^*_{l'})\right)\\
=\mathfrak M(Id)W_lZ+\sum_{\substack{\tau\in\Gal(M/\Q)\\ \tau\ne Id}}\mathfrak M(\tau)\tau(W_l)\tau(Z).
\end{multline}
Substitute $l=1$:
\begin{equation}\label{el.finalapprox.1}
\Lambda_1=\mathfrak M(Id)W_1Z+\sum_{\substack{\tau\in\Gal(M/\Q)\\ \tau\ne Id}}\mathfrak M(\tau)\tau(W_1)\tau(Z).
\end{equation}
Using the definition of $C$ and the bound for $\tau(Z)$, we obtain
$$
|\Lambda_1-\mathfrak M(Id)W_1Z|\le\frac{C\Delta}{Z^{\frac1{m-1}}}\le C\Delta.
$$
This proves the first assertion.

Assume that $|\Lambda_1|>C\Delta$. Then
$$
|\mathfrak M(Id)W_1|Z\ge|\Lambda_1|-C\Delta.
$$
Therefore, $\mathfrak M(Id)\ne0$ and
\begin{equation}\label{el.finalapprox.z}
Z\ge\frac{|\Lambda_1|-C\Delta}{|\mathfrak M(Id)W_1|}.
\end{equation}
Multiply the equality \eqref{el.finalapprox.1} by $\frac{W_l}{W_1}$ and subtract from \eqref{el.finalapprox.main}.
Then use the definition of $C_l$ and the bound for $\tau(Z)$:
$$
\left|\Lambda_l-\frac{W_l}{W_1}\Lambda_1\right|\le C_l\frac{\Delta}{Z^{\frac1{m-1}}}.
$$
Divide the last inequality by $|\Lambda_1|$:
$$
\left|\frac{\Lambda_l}{\Lambda_1}-\frac{W_l}{W_1}\right|\le C_l\frac{\Delta}{|\Lambda_1|Z^{\frac1{m-1}}}.
$$
Now it is sufficient to use \eqref{el.finalapprox.z} to conclude the proof.
\end{proof}

The article \cite{peck} uses a knowledge of group of units in $\ro^*_M$ (Dirichlet theorem)
and looks for $\sum_{i=1}^m \Lambda_iW_i^*$ as a unit of a special form. It allows to prove
interesting theoretical results, but it is quite inconvenient from the practical point of view.
We use another approach.

We want to construct simultaneous approximations to elements of the field $\mathcal M=K_G\cap\R$.
In order to do this, we apply Theorem \ref{el.finalapprox} to the field $M=\mathcal M$.
Thus, $m=[\mathcal M:\Q]=2^{t-1}$, $t\ge2$, and $\Gal(\mathcal M/\Q)$ consists of automorphisms
$\tau_\lambda$ defined by \eqref{el.taudef}, $\lambda\in\{0,1\}^{t-1}$.

It is convenient to numerate sets related to the field $\mathcal M$ by vectors from $\{0,1\}^{t-1}$.
Hereafter we assume that two basises $\omega_\mu$ and $\omega^*_\mu$ of $\mathcal M$ over $\Q$
and a function $\mathfrak M:\Gal(\mathcal M/\Q)\to\R$ are given and satisfy the following
conditions:
\begin{enumerate}
\item $\omega^*_{0,\dots,0}=1$.
\item Any element of $\ro_M$ is a linear combination of $\{\omega_\mu^*\}$ with integer coefficients.
\item For any $\lambda,\lambda'\in\{0,1\}^{t-1}$,
\begin{equation}\label{el.dual}
\sum_{\mu\in\{0,1\}^{t-1}}\mathfrak M(\tau_\mu)\tau_\mu\left(\omega_\lambda\omega^*_{\lambda'}\right)=
\begin{cases}
1,&\mbox{if }\lambda=\lambda',\\
0,&\mbox{if }\lambda\ne\lambda'.
\end{cases}
\end{equation}
\end{enumerate}
We call such a pair an $\mathfrak M$-pair. It is easy to see that these conditions imply
conditions on basises from Theorem \ref{el.finalapprox} applied to the numbers
\begin{eqnarray*}
W_{1+\mu_1+2\mu_2+2^2\mu_3+\ldots+2^{t-2}\mu_{t-1}}&=&\omega_\mu,\\
W^*_{1+\mu_1+2\mu_2+2^2\mu_3+\ldots+2^{t-2}\mu_{t-1}}&=&\omega^*_\mu.
\end{eqnarray*}

Note that if $x\in\ro_{\mathcal M}$, then $x\beta_{0,\dots,0}^*\in\ro_{K_G}\cap i\R$.
Two following corollaries follow easily from Theorems \ref{el.basiscapr}--\ref{el.basiscaprend}.
As in these theorems, the value of $\sqrt{d}$ is chosen as the product $\sqrt{q_1^*}\ldots\sqrt{q_t^*}$.

\begin{corollary} Conditions 1--3 hold for
\begin{equation}\label{el.omega1def}\begin{array}{rcl}
\omega_{\mu_1,\dots,\mu_{t-1}}&=&\frac{\beta_{\mu_1,\dots,\mu_{t-1}}}{\beta_{0,\dots,0}},\\
\omega^*_{\mu_1,\dots,\mu_{t-1}}&=&\frac{\beta^*_{\mu_1,\dots,\mu_{t-1}}}{\beta^*_{0,\dots,0}},\\
\mathfrak M(\tau_{\mu_1,\dots,\mu_{t-1}})&=&(-1)^{\mu_1+\ldots+\mu_{t-1}}\frac{\tau_{\mu_1,\dots,\mu_{t-1}}\left(\beta_{0,\dots,0}\beta^*_{0,\dots,0}\right)}{\sqrt{d}}.
\end{array}\end{equation}
\end{corollary}

\begin{corollary} Conditions 1--3 hold for
\begin{equation}\label{el.omega2def}\begin{array}{rcl}
\omega_{\mu_1,\dots,\mu_{t-1}}&=&\frac{\beta^*_{\mu_1,\dots,\mu_{t-1}}}{\beta^*_{0,\dots,0}},\\
\omega^*_{\mu_1,\dots,\mu_{t-1}}&=&\frac{\beta_{\mu_1,\dots,\mu_{t-1}}}{\beta_{0,\dots,0}},\\
\mathfrak M(\tau_{\mu_1,\dots,\mu_{t-1}})&=&(-1)^{\mu_1+\ldots+\mu_{t-1}}\frac{\tau_{\mu_1,\dots,\mu_{t-1}}\left(\beta_{0,\dots,0}\beta^*_{0,\dots,0}\right)}{\sqrt{d}}.
\end{array}\end{equation}
\end{corollary}

Theorem \ref{el.finalapprox} also uses integer numbers $\Lambda_i$ and a constant $\Delta$. The rest
of this section deals with construction of a set $A_\mu$ such that the numbers
$$
\Lambda_{1+\mu_1+2\mu_2+2^2\mu_3+\ldots+2^{t-2}\mu_{t-1}}=A_{\mu_1,\dots,\mu_{t-1}}
$$
satisfy the assumption of Theorem \ref{el.finalapprox} with some $\Delta$.

We need the following quantities to describe the algorithm. Let $\lambda\in\{0,1\}^{t-1}$, $\lambda\ne(0,\dots,0)$.
Define
$$
\delta_\lambda=(q_1^*)^{\lambda_1}\ldots(q_{t-1}^*)^{\lambda_{t-1}}(q_t^*)^{\lambda_{u+1}\oplus\ldots\oplus\lambda_{t-1}}.
$$
If $\delta_\lambda$ is even, set
$$
g_\lambda=\frac{\sqrt{\delta_\lambda}}2,
$$
otherwise set
$$
g_\lambda=\frac{1+\sqrt{\delta_\lambda}}2.
$$
Then $g_\lambda\in\ro_M$.

We use continued fractions. We remind that for any number $X\in\R$ two sequences are defined:
\textit{complete quotients} $X_0,X_1,X_2,\dots$ and \textit{partial quotients} $a_0,a_1,a_2,\dots$,
where $X_0=X$, $a_n=\lfloor X_n\rfloor$, $X_{n+1}=\frac1{X_n-a_n}$. These sequences are finite
(i.e. $X_n$ is indefinite for some $n$) if and only if $X\in\Q$. In addition, the sequence
of \textit{convergents} $\frac{P_0}{Q_0},\frac{P_1}{Q_1},\frac{P_2}{Q_2},\dots$
is defined as follows: $P_{-1}=0,Q_{-1}=0,P_0=a_0,Q_0=1,P_{n+1}=a_{n+1}P_n+P_{n-1},Q_{n+1}=a_{n+1}Q_n+Q_{n-1}$.
It is well known (e.g. \cite[Theorems 9 and 12]{hinchin}), that for any $n\ge0$
\begin{equation}\label{el.contapprox}
\left|X-\frac{P_n}{Q_n}\right|<\frac1{Q_nQ_{n+1}},\mbox{ if }X_{n+2}\mbox{ is defined};
\end{equation}
\begin{equation}\label{el.qlower}
Q_n\ge2^{\frac{n-1}2}.
\end{equation}

In the case of quadratic irrationals these sequences have an additional structure.
We use some results from \cite[\S II.10]{venkov} collected in the next statement.
\begin{statement}\label{el.venkov} Let $a,b,c$ be integer numbers with $\gcd(a,b,c)=1$. Let $\delta=b^2-ac>0$
be not an exact square. We call the roots of the equation $ax^2+2bx+c=0$ as
\textit{irrationals of determinant $\delta$}.

Let $X=\frac{-b+\sqrt\delta}a$ be an irrational of determinant $\delta$.
Then all complete quotients $X_n$ are also irrationals of determinant $\delta$ and
have a form $X_n=\frac{x_n+\sqrt\delta}{y_n}$, where $x_n,y_n\in\Z$ are uniquely determined.
Let $a_n=\lfloor X_n\rfloor$ be partial quotients for $X$.
Define $y_{-1}=-c=\frac{\delta-b^2}a\in\Z$. The following recurrent formulas hold:
\begin{equation}\label{el.xyrecur}\begin{array}{rcl}
x_n&=&y_{n-1}a_{n-1}-x_{n-1},\quad n\ge1;\\
\delta&=&x_n^2+y_ny_{n-1},\quad n\ge0;\\
y_n&=&y_{n-2}-a_{n-1}(x_n-x_{n-1}),\quad n\ge1.
\end{array}\end{equation}
Moreover, for $n\ge0$
$$
X_1\ldots X_n=\frac{(-1)^n}{P_{n-1}-Q_{n-1}X};
$$
\begin{equation}\label{el.converg}
aP_{n-1}^2+2bP_{n-1}Q_{n-1}+cQ_{n-1}^2=(-1)^ny_n.
\end{equation}
A number $\frac{x+\sqrt\delta}y$ with $x,y\in\Z$ is \textit{reduced}
if $\frac{x+\sqrt\delta}y>1$ and $-1<\frac{x-\sqrt\delta}y<0$. A number
$\frac{x+\sqrt\delta}y$ is reduced if and only if $0<\sqrt\delta-x<y<\sqrt\delta+x$.
If $X$ is reduced, then all complete quotients for $X$ are also reduced.
\end{statement}

We calculate continued fractions for all numbers $g_\lambda$ in parallel,
$\lambda\in\{0,1\}^{t-1}$, $\lambda\ne0$. Let $X_{\lambda,n}$ be complete
quotients for $g_\lambda$, $a_{\lambda,n}$ be partial quotients for $g_\lambda$.
Let $P_{\lambda,n}$ and $Q_{\lambda,n}$ be numerators and denominators
of convergents of $g_\lambda$ respectively. Let $x_{\lambda,n}$,
$y_{\lambda,n}$ be the quantities $x_n$, $y_n$
from Statement \ref{el.venkov} calculated for $X=g_\lambda$.
Let $\sigma_\lambda$ denote the only nontrivial automorphism of the field $\Q(g_\lambda)$.

If $\delta_\lambda$ is odd, then $g_\lambda$ is an irrational of determinant $\delta$,
$x_{\lambda,0}=1$, $y_{\lambda,0}=2$, $y_{\lambda,-1}=\frac{\delta_\lambda-1}2$.
It is easy to see from \eqref{el.xyrecur} by induction that $x_{\lambda,n}$ is odd
and $y_{\lambda,n}$ is even for all $n$. Let $x'_{\lambda,n}=\frac{x_{\lambda,n}-1}2\in\Z$
and $y'_{\lambda,n}=\frac{y_{\lambda,n}}2\in\Z$. The quadratic polynomial $ax^2+2bx+c$,
where $a,b,c$ are defined in Statement \ref{el.venkov}, has the first coefficient
2 and roots $g_\lambda,\sigma_\lambda(g_\lambda)$. Thus,
\eqref{el.converg} is equivalent to
$2(P_{\lambda,n-1}-Q_{\lambda,n-1}g_\lambda)\sigma_\lambda(P_{\lambda,n-1}-Q_{\lambda,n-1}g_\lambda)
=(-1)^ny_{\lambda,n}=(-1)^n2y'_{\lambda,n}$.

If $\delta_\lambda$ is even, then $g_\lambda$ is an irrational of determinant $\frac{\delta_\lambda}4$,
$x_{\lambda,0}=0,y_{\lambda,0}=1,y_{\lambda,-1}=\frac{\delta_\lambda}4$. Let $x'_{\lambda,n}=x_{\lambda,n}$
and $y'_{\lambda,n}=y_{\lambda,n}$. The quadratic polynomial $ax^2+2bx+c$, where $a,b,c$ are defined
in Statement \ref{el.venkov}, has the first coefficient 1 and roots $g_\lambda,\sigma_\lambda(g_\lambda)$.
Thus, \eqref{el.converg} is equivalent to
$(P_{\lambda,n-1}-Q_{\lambda,n-1}g_\lambda)\sigma_\lambda(P_{\lambda,n-1}-Q_{\lambda,n-1}g_\lambda)
=(-1)^ny_{\lambda,n}=(-1)^ny'_{\lambda,n}$.

In both cases
$$X_{\lambda,n}=\frac{g_\lambda+x'_{\lambda,n}}{y'_{\lambda,n}};$$
\begin{equation}\label{el.quasizandcong}
(P_{\lambda,n-1}-Q_{\lambda,n-1}g_\lambda)\sigma_\lambda(P_{\lambda,n-1}-Q_{\lambda,n-1}g_\lambda)=(-1)^ny'_{\lambda,n}.
\end{equation}

Statement \ref{el.venkov} gives an efficient method to calculate numbers $x'_{\lambda,n}$, $y'_{\lambda,n}$,
$a_{\lambda,n}=\lfloor X_{\lambda,n}\rfloor$ in sequence and then $P_{\lambda,n}$ and $Q_{\lambda,n}$.
The algorithm uses numbers $x'_{\lambda,n}$, $y'_{\lambda,n}$ and
\begin{equation}\label{el.zdef}
z_{\lambda,n}=\frac1{X_{\lambda,1}\ldots X_{\lambda,n}}=(-1)^n(P_{\lambda,n-1}-Q_{\lambda,n-1}g_\lambda)\in\ro_{\mathcal M}.
\end{equation}
This definition and the equality \eqref{el.quasizandcong} imply that for any $n\ge0$
\begin{equation}\label{zandcong}
z_{\lambda,n}\sigma_{\lambda}(z_{\lambda,n})=(-1)^ny'_{\lambda,n}.
\end{equation}

Numbers $A_\mu$ are taken from the equality
$$
\prod_{\lambda\ne0}\left((-1)^{n_\lambda}\sigma_\lambda(z_{\lambda,n_\lambda})\right)=\sum_\mu A_\mu\omega_\mu^*.
$$
The left-hand side is the product of algebraic integers due to \eqref{el.quasizandcong},
so the condition 2 on $\mathfrak M$-pair guarantees that $A_\mu$ are integers.

Each step of the algorithm increments exactly one of numbers $n_\lambda$. This multiplies
$\sum_\mu A_\mu\omega^*_\mu$ by
$$
\frac{(-1)^{n+1}\sigma_\lambda\left(z_{\lambda,n+1}\right)}{(-1)^n\sigma_\lambda\left(z_{\lambda,n}\right)}=
\frac{y'_{\lambda,n+1}/z_{\lambda,n+1}}{y'_{\lambda,n}/z_{\lambda,n}}=\frac{X_{\lambda,n+1}y'_{\lambda,n+1}}{y'_{\lambda,n}}=
\frac{g_\lambda+x'_{\lambda,n+1}}{y'_{\lambda,n}}.
$$
Thus, we need to switch from the set $A_\mu$ to the set $A'_\mu$ such that
$$
\left(\sum_\mu A'_\mu\omega_\mu^*\right)=\left(\sum_\xi A_\xi\omega_\xi^*\right)\frac{g_\lambda+x_{\lambda}}{y_{\lambda}}
$$
(where $x_\lambda=x'_{\lambda,n_{\lambda+1}}$ and $y_\lambda=y'_{\lambda,n_\lambda}$).
Since $\{\omega^*_\mu\}$ is a $\Q$-basis of $\mathcal M$ and $g_\mu\in\mathcal M$,
we can precompute numbers $c_{\mu\xi\eta}\in\Q$ such that
$$
\omega^*_\xi g_\eta=\sum_\mu c_{\mu\xi\eta}\omega^*_\mu.
$$
On each step we calculate
$$
\left(\sum_\xi A_\xi\omega^*_\xi\right)\frac{g_\lambda+x_\lambda}{y_\lambda}=
\frac1{y_\lambda}\left(\sum_\xi A_\xi\sum_\mu c_{\mu\xi\lambda}\omega^*_\mu+\sum_\xi A_\xi\omega^*_\xi x_\lambda\right)
=\sum_\mu\frac{\sum_\xi A_\xi c_{\mu\xi\lambda}+A_\mu x_\lambda}{y_\lambda}\omega^*_\mu.
$$

Now we are ready to show the algorithm.

\textbf{Algorithm for construction of simultaneous approximations.}
Input data: the sets $\delta_\lambda$, $g_\lambda$, $c_{\mu\xi\eta}$ as above, the threshold $N_0>0$.
Output data: the set of integer numbers $A_\mu$ such that $|A_{0,\dots,0}|\ge N_0$ and
$\frac{A_\mu}{A_{0,\dots,0}}$ is an approximation to $\frac{\omega_\mu}{\omega_{0,\dots,0}}$ for
each $\mu\in\{0,1\}^{t-1}$.

The algorithm keeps a set of $2^{t-1}$ integer numbers $A_\mu$ and auxiliary sets of
non-negative integers $x_\lambda$, positive integers $(y_\lambda,\tilde y_\lambda)$
and positive reals $(z_\lambda,\tilde z_\lambda)$ for $\lambda\in\{0,1\}^{t-1}$,
$\lambda\ne(0,\dots,0)$. These sets have the following sense: if each vector $\lambda$
was selected $n_\lambda$ times during the step 3 below, then
\begin{eqnarray*}
x_\lambda&=&x'_{\lambda,n_\lambda},\\
(y_\lambda,\tilde y_\lambda)&=&(y'_{\lambda,n_\lambda},y'_{\lambda,n_\lambda-1}),\\
(z_\lambda,\tilde z_\lambda)&=&(z_{\lambda,n_\lambda},z_{\lambda,n_\lambda-1}),\\
\sum_\mu A_\mu\omega^*_\mu&=&\prod_{\lambda\ne0}((-1)^{n_\lambda}\sigma_\lambda(z_{\lambda,n_\lambda})).
\end{eqnarray*}
The algorithm consists of the following steps.
\begin{enumerate}
\item \textit{Initialization.} For each $\lambda\in\{0,1\}^{t-1}$, $\lambda\ne(0,\dots,0)$ set
\begin{eqnarray*}
A_{0,\dots,0}&:=&1\\
A_\lambda&:=&0\\
x_\lambda&:=&0\\
(y_\lambda,\tilde y_\lambda)&:=&(1,{\scriptstyle\left\lfloor\frac{\delta_\lambda}4\right\rfloor})\\
(z_\lambda,\tilde z_\lambda)&:=&(1,g_\lambda).
\end{eqnarray*}
\item \textit{Iterations.} Repeat the following steps while $|A_{0,\dots,0}|<N_0$.
\item Select any $\lambda$ such that $z_\lambda=\max_{\mu\ne(0,\dots,0)} z_\mu$.
\item Calculate $a=\left\lfloor\frac{g_\lambda+x_\lambda}{y_{\lambda}}\right\rfloor$.
\item Set $(z_\lambda,\tilde z_\lambda):=(\tilde z_\lambda-az_\lambda,z_\lambda)$.
\item Save $x=x_\lambda$. Set $x_\lambda:=ay_\lambda-x_\lambda-4\left\{\frac{\delta_\lambda}4\right\}$.
Set $(y_\lambda,\tilde y_\lambda):=(\tilde y_\lambda-a(x_\lambda-x),y_\lambda)$.
(As shown below, the new value of $x_\lambda$ is always a non-negative integer, the new value of
$y_\lambda$ is always a positive integer.)
\item For each $\mu$ calculate
$$
A'_\mu=\frac{\sum_\xi A_\xi c_{\mu\xi\lambda}+A_\mu x_\lambda}{\tilde y_\lambda}.
$$
(As shown above, $A'_\mu\in\Z$ for all $\mu$.) Set $A_\mu:=A'_\mu$.
\end{enumerate}

\begin{theorem}\label{el.approxalg} The algorithm completes in $O(\ln N_0)$ steps.
The following inequalities hold in every step of the algorithm:
$$
0\le x_\lambda<\sqrt{\delta_\lambda}-g_\lambda,
$$
$$
0<y_\lambda<\sqrt{\delta_\lambda};
$$
$$
Z=\sum_\mu A_\mu\omega^*_\mu\ge1,
$$
$$
\left|\tau_\lambda\left(\sum_\mu A_\mu\omega^*_\mu\right)\right|\le\frac{\sqrt{|d|}^m}{Z^{\frac1{m-1}}}
\mbox{ for }\lambda\ne(0,\dots,0).
$$
\end{theorem}
\begin{proof} We start from the bounds for $x'_{\lambda,n}$, $y'_{\lambda,n}$.
\begin{lemma}\label{xybound} Let $\lambda\in\{0,1\}^{t-1}$, $\lambda\ne0$. Let $n\ge1$ be an integer.
Then
$$
0\le x'_{\lambda,n}<\sqrt{\delta_\lambda}-g_\lambda,
$$
$$
0<y'_{\lambda,n}<\sqrt{\delta_\lambda},
$$
$$
-1<\sigma_\lambda(X_{\lambda,n})<0.
$$
\end{lemma}
\begin{proof}
Assume first that $\delta_\lambda$ is odd. By definition, $X_{\lambda,1}=\frac1{g_\lambda-\lfloor g_\lambda\rfloor}$.
Obviously, $X_{\lambda,1}>1$. In addition, $\sigma_\lambda(X_{\lambda,1})=\frac1{1-g_\lambda-\lfloor g_\lambda\rfloor}$
and $g_\lambda>1$ imply that $-1<\sigma_\lambda(X_{\lambda,1})<0$. Therefore, due to Statement \ref{el.venkov}
all complete quotients of $g_\lambda$ starting from $X_{\lambda,1}$ are reduced irrationals of determinant $\delta_\lambda$.
That is, $0<\sqrt{\delta_\lambda}-x_{\lambda,n}<y_{\lambda,n}<\sqrt{\delta_\lambda}+x_{\lambda,n}$ for $n\ge1$.
Since $x'_{\lambda,n}=\frac{x_{\lambda,n}-1}2$ and $y'_{\lambda,n}=\frac{y_{\lambda,n}}2$ in this case,
we obtain the required bounds.

Assume now that $\delta_\lambda$ is even. As in the first case, $X_{\lambda,1}=\frac1{g_\lambda-\lfloor g_\lambda\rfloor}>1$.
In addition, $\sigma_\lambda(X_{\lambda,1})=-\frac1{g_\lambda+\lfloor g_\lambda\rfloor}$ and $g_\lambda>1$ imply
that $-1<\sigma_\lambda(X_{\lambda,1})<0$. Therefore, due to Statement \ref{el.venkov} all complete quotients of $g_\lambda$
starting from $X_{\lambda,1}$ are reduced irrationals of determinant $\frac{\delta_\lambda}4$. That is,
$0<\frac{\sqrt{\delta_\lambda}}2-x_{\lambda,n}<y_{\lambda,n}<\frac{\sqrt{\delta_\lambda}}2+x_{\lambda,n}$ for $n\ge1$.
Since $x'_{\lambda,n}=x_{\lambda,n}$ and $y'_{\lambda,n}=y_{\lambda,n}$ in this case, we obtain
the required bounds.
\end{proof}

Since $X_{\lambda,n}=\frac{g_\lambda+x'_{\lambda,n}}{y'_{\lambda,n}}$, Lemma \ref{xybound} immediately implies

\begin{corollary} For $n\ge1$
\begin{equation}\label{el.abound}
X_{\lambda,n}<\sqrt{\delta_\lambda}.
\end{equation}
\end{corollary}

Let $n_\lambda$ denote the number of times when $\lambda$ was selected in the step 3 of the algorithm, $\lambda\ne0$.

The inequality $Z=\prod_{\lambda\ne0}((-1)^{n_\lambda}\sigma_\lambda(z_{\lambda,n_\lambda}))\ge1$
follows immediately from the last inequality of Lemma \ref{xybound} and the definition
$z_{\lambda,n_\lambda}=\frac1{X_{\lambda,1}\dots X_{\lambda,n_\lambda}}$.

\begin{lemma}\label{el.zbalanced}
$$
\frac{\max_{\mu\ne(0,\dots,0)}z_{\mu,n_\mu}}{\min_{\mu\ne(0,\dots,0)}z_{\mu,n_\mu}}\le\sqrt{|d|}.
$$
\end{lemma}
\begin{proof}
Before iterations the left-hand side equals 1, so the inequality holds.
Assume that the inequality holds after some number of iterations. Assume that
the step 3 of the next iteration selects the value $\lambda$, i.e.
$$
z_{\lambda,n_\lambda}=\max_{\mu\ne(0,\dots,0)}z_{\mu,n_\mu}.
$$
Let $n'_\lambda=n_\lambda+1$ and $n'_\mu=n_\mu$ for $\mu\ne\lambda$, $\mu\ne(0,\dots,0)$.
Obviously, $X_{\lambda,n'_\lambda}>1$, so $z_{\lambda,n'_\lambda}<z_{\lambda,n_\lambda}$.
There are two possible cases:
\begin{itemize}
\item $z_{\lambda,n'_\lambda}\ge\min_{\mu\ne(0,\dots,0)}z_{\mu,n_\mu}$. In this case
$$\min_{\mu\ne(0,\dots,0)}z_{\mu,n'_\mu}=\min_{\mu\ne(0,\dots,0)}z_{\mu,n_\mu},$$ therefore,
$$
\frac{\max_{\mu\ne(0,\dots,0)}z_{\mu,n'_\mu}}{\min_{\mu\ne(0,\dots,0)}z_{\mu,n'_\mu}}\le\frac{\max_{\mu\ne(0,\dots,0)}z_{\mu,n_\mu}}{\min_{\mu\ne(0,\dots,0)}z_{\mu,n_\mu}}
\le\sqrt{|d|}.
$$
\item $z_{\lambda,n'_\lambda}<\min_{\mu\ne(0,\dots,0)}z_{\mu,n_\mu}$. In this case
$\min_{\mu\ne(0,\dots,0)}z_{\mu,n'_\mu}=z_{\lambda,n'_\lambda}$; using \eqref{el.abound}, we obtain
$$
\frac{\max_{\mu\ne(0,\dots,0)}z_{\mu,n'_\mu}}{\min_{\mu\ne(0,\dots,0)}z_{\mu,n'_\mu}}\le\frac{z_{\lambda,n_\lambda}}{z_{\lambda,n'_\lambda}}=X_{\lambda,n_\lambda+1}
<\sqrt{\delta_\lambda}\le\sqrt{|d|}.
$$
\end{itemize}
\end{proof}

We recall that $\sigma_\lambda$ is an automorphism of the field $\Q(g_\lambda)\subset\mathcal M$.
Note that for any $\lambda$ and $\mu$ the automorphism $\tau_\mu$ can be restricted to the field $\Q(g_\lambda)$.
Since
\begin{multline*}
\tau_\mu\left(\sqrt{(q_1^*)^{\lambda_1}\ldots(q_{t-1}^*)^{\lambda_{t-1}}(q_t^*)^{\lambda_{u+1}\oplus\ldots\oplus\lambda_{t-1}}}\right)\\
=((-1)^{\mu_1}\sqrt{q_1^*})^{\lambda_1}\ldots((-1)^{\mu_{t-1}}\sqrt{q_{t-1}^*})^{\lambda_{t-1}}\sqrt{q_t^*}^{\lambda_{u+1}\oplus\ldots\oplus\lambda_{t-1}}\\
=(-1)^{\sum_{i=1}^{t-1}\lambda_i\mu_i}\sqrt{(q_1^*)^{\lambda_1}\ldots(q_{t-1}^*)^{\lambda_{t-1}}(q_t^*)^{\lambda_{u+1}\oplus\ldots\oplus\lambda_{t-1}}},
\end{multline*}
the restriction $\tau_\mu|_{\Q(g_\lambda)}$ acts trivially if $\sum_{i=1}^{t-1}\lambda_i\mu_i\equiv0\pmod2$
and coincides with $\sigma_\lambda$ otherwise.

Let $\max_{\mu\ne(0,\dots,0)}z_{\mu,n_\mu}=\varepsilon$. Lemma \ref{el.zbalanced} implies that
$$\frac\varepsilon{\sqrt{|d|}}\le z_{\lambda,n_\lambda}\le\varepsilon$$
for each $\lambda\ne(0,\dots,0)$. Equalities \eqref{el.zdef}, \eqref{zandcong} and Lemma \ref{xybound}
imply that $1\le z_{\lambda,n_\lambda}
\left|\sigma_\lambda\left(z_{\lambda,n_\lambda}\right)\right|<\sqrt{\delta_\lambda}\le\sqrt{|d|}$.
Thus,
$$
\frac1\varepsilon\le\left|\sigma_\lambda\left(z_{\lambda,n_\lambda}\right)\right|
\le\frac{|d|}\varepsilon.
$$
By construction,
$$
Z=\prod_{\lambda\ne0}\left|\sigma_\lambda\left(z_{\lambda,n_\lambda}\right)\right|
\le\left(\frac{|d|}\varepsilon\right)^{m-1},
$$
so
$$
\varepsilon\le\frac{|d|}{Z^{\frac1{m-1}}}.
$$
Let $\lambda\ne0$. The condition $\sum_{i=1}^{t-1}\lambda_i\mu_i\equiv0\pmod2$ as an equation for $\mu\in\{0,1\}^{t-1}$
has exactly $\frac n2$ solutions, including zero.
\begin{multline*}
\left|\tau_\lambda\left(\sum_\mu A_\mu\omega^*_\mu\right)\right|=
\prod_{2\mid\sum_i\lambda_i\mu_i,\mu\ne0}\left|\sigma_\mu\left(z_{\mu,n_\mu}\right)\right|
\cdot\prod_{2\nmid\sum_i\lambda_i\mu_i}\left|z_{\mu,n_\mu}\right|\\
\le\left(\frac{|d|}\varepsilon\right)^{\frac m2-1}\varepsilon^{\frac m2}=
|d|^{\frac m2-1}\varepsilon\le\frac{\sqrt{|d|}^m}{Z^{\frac1{m-1}}}.
\end{multline*}
It remains to show that the algorithm completes in $O(\ln N_0)$ iterations.
A part of theorem which is already proved allows to apply Theorem \ref{el.finalapprox}.
Thus, the following inequality holds in any step of the algorithm:
$$
|A_{0,\dots,0}|\ge|\mathfrak M(Id)\omega_{0,\dots,0}|Z-C\sqrt{|d|}^m,
$$
where constants $\mathfrak M(Id)\omega_{0,\dots,0}\ne0$ and $C\sqrt{|d|}^m$
depend only on basises.

Now \eqref{zandcong} implies
$$
Z=\prod_{\lambda\ne0}\frac{y'_{\lambda,n_\lambda}}{z_{\lambda,n_\lambda}}\ge\left(\prod_{\lambda\ne0}z_{\lambda,n_\lambda}\right)^{-1},
$$
with \eqref{el.zdef}, \eqref{el.contapprox} and \eqref{el.qlower} this yields
$$
Z\ge\left(\prod_{\lambda\ne0}\left|P_{\lambda,n_\lambda-1}-Q_{\lambda,n_{\lambda-1}}g_\lambda\right|\right)^{-1}\ge
\prod_{\lambda\ne0}Q_{\lambda,n_\lambda}\ge\prod_{\lambda\ne0}2^{\frac{n_\lambda-1}2}=2^{\frac{\sum_{\lambda\ne0}n_\lambda-(m-1)}2}.
$$
The sum $\sum_{\lambda\ne0}n_\lambda$ is the number of algorithm iterations. Thus, after $O(\ln N_0)$ iterations
the following inequality is reached:
$$
Z\ge\frac{N_0+C\sqrt{|d|}^m}{|\mathfrak M(Id)\omega_{0,\dots,0}|}.
$$
This implies $|A_{0,\dots,0}|\ge N_0$ and concludes the proof.
\end{proof}

\section{Calculation of an algebraic integer by its approximation}\label{el.get}

We want to calculate numbers $b_\mu\in\Z$ by an approximate value of $\sum_\mu b_\mu\beta_\mu$,
and also numbers $b'_\mu\in\Z$ by an approximate value of $\sum_\mu b'_\mu\beta^*_\mu$.
Section \ref{el.upbound} gives apriori bounds of the form
\begin{equation}\label{el.finalhighbound}\begin{array}{rcl}
\left|\tau_\lambda\left(\sum_\mu b_\mu\beta_\mu\right)\right|&\leq& T_0,\\
\left|\tau_\lambda\left(\sum_\mu b'_\mu\beta^*_\mu\right)\right|&\leq& T_0,
\end{array}\end{equation}
where $T_0$ depends only on $D$.
Section \ref{el.simapprox} gives a set of simultaneous approximations
to the numbers $\frac{\beta_\mu}{\beta_{0,\dots,0}}$ and another set
for the numbers $\frac{\beta^*_\mu}{\beta^*_{0,\dots,0}}$. The precision
of these approximations depends on a parameter $N_0$.

Approximations constructed in Section \ref{el.simapprox} satisfy
Theorem \ref{el.approxalg} which will be used.
(One can prove that any simultaneous approximations $\Lambda_i$
to a basis $W_i$ with a bound of the form
$\left|\frac{\Lambda_i}{\Lambda_1}-\frac{W_i}{W_1}\right|\le\frac{C'_i}{|\Lambda_1|^{1+\alpha}}$
satisfy the last bound from Theorem \ref{el.approxalg} with an exponent $\alpha$
instead of $\frac1{m-1}$. Thus, actually any sufficiently good approximations
can be used.)

We continue to use the basises $\omega_\mu$, $\omega^*_\mu$ and the function
$\mathfrak M$ defined in \eqref{el.omega1def} (for $b_\mu$) or \eqref{el.omega2def}
(for $b^*_\mu$). It is easy to see that they satisfy the following property
additionally to properties 1--3 of $\mathfrak M$-pairs:
\begin{enumerate}
\item[2'.] If $x\in\ro_{\mathcal M}$, then $\omega_\xi x$ is a linear combination of $\{\omega_\mu\}$ with integer coefficients.
\end{enumerate}

For definiteness, we show how to find $b_\mu$; the method for $b'_\mu$ is analogous.

Let $X_\eta\in\ro_{\mathcal M}$ be a set of $m=2^{t-1}$ numbers linearly independent over $\Q$.
For example, one possible choice is $X_\eta=\beta_\eta$; another possible choice is $X_{0,\dots,0}=1$
and $X_\eta=g_\eta$ for $\eta\ne(0,\dots,0)$. The property 2' implies that
\begin{equation}\label{el.xmultdef}
\omega_\xi X_\eta=\sum_\mu x_{\mu\xi\eta}\omega_\mu,
\end{equation}
with $x_{\mu\xi\eta}\in\Z$. (The choice $X_\eta=g_\eta$ is convenient in that $x_{\mu\xi\eta}$
are the same as $c_{\mu\xi\eta}$ with transposed $\beta_\mu$ and $\beta_\mu^*$. The choice
$X_\eta=\beta_\eta$ results in numbers $x_{\mu\xi\eta}$ which are slightly less in the absolute value.)

Assume that the precision $\varepsilon$ is selected. We know the value of the sum $\sum_\xi b_\xi\beta_\xi$
with the precision $\varepsilon$; in other words, we know a number $\gamma$ such that
$\left|\sum_\xi b_\xi\beta_\xi-\gamma\right|<\varepsilon$. Divide this inequality by
$\beta_{0,\dots,0}$ and multiply by $X_\eta$.
$$
\left|\sum_\xi b_\xi\omega_\xi X_\eta-\frac{\gamma X_\eta}{\beta_{0,\dots,0}}\right|\le\frac{\varepsilon|X_\eta|}{|\beta_{0,\dots,0}|},
$$
\begin{equation}\label{el.bifzero}
\left|\sum_\mu\left(\sum_\xi b_\xi x_{\mu\xi\eta}\right)\omega_\mu-\frac{\gamma X_\eta}{\beta_{0,\dots,0}}\right|\le
\frac{\varepsilon|X_\eta|}{|\beta_{0,\dots,0}|}.
\end{equation}
Let $B_{\mu\eta}=\sum_\xi b_\xi x_{\mu\xi\eta}\in\Z$. For any $\mu'$ we have from \eqref{el.dual} that
$$
B_{\mu'\eta}=\sum_\mu B_{\mu\eta}\sum_{\lambda}\mathfrak M(\tau_\lambda)\tau_\lambda(\omega_\mu\omega^*_{\mu'})=
\sum_{\lambda}\mathfrak M(\tau_\lambda)\tau_\lambda(\omega^*_{\mu'})\tau_\lambda\left(\sum_\mu B_{\mu\eta}\omega_\mu\right);
$$
\begin{equation}\label{el.almostsys}
\sum_{\mu'}A_{\mu'}B_{\mu'\eta}=\sum_\lambda\mathfrak M(\tau_\lambda)\tau_\lambda\left(\sum_{\mu'}A_{\mu'}\omega^*_{\mu'}\right)\tau_\lambda\left(\sum_\mu B_{\mu\eta}\omega_\mu\right).
\end{equation}
The term with $\lambda=0$ is special. In this case \eqref{el.bifzero} gives an approximate value of the last factor
with a bound for approximation error. Now consider $\lambda\ne0$.
Theorem \ref{el.approxalg} gives a bound for the second factor.
$$
\tau_\lambda\left(\sum_\mu B_{\mu\eta}\omega_\mu\right)=\tau_\lambda\left(\sum_\xi b_\xi\sum_\mu x_{\mu\xi\eta}\omega_\mu\right)=
\tau_\lambda\left(\sum_\xi b_\xi\omega_\xi\right)\tau_\lambda(X_\eta),
$$
with \eqref{el.finalhighbound} this implies that
$$
\left|\tau_\lambda\left(\sum_\mu B_{\mu\eta}\omega_\mu\right)\right|\leq T_0\left|\tau_\lambda(X_\eta)\right|.
$$
Therefore, \eqref{el.almostsys}, \eqref{el.bifzero} and Theorem \ref{el.approxalg} imply that
\begin{equation}\label{el.almostsys2}
\left|\sum_{\mu'}A_{\mu'}B_{\mu'\eta}-\mathfrak M(Id)Z\frac{\gamma X_\eta}{\beta_{0,\dots,0}}\right|
\le\left|\mathfrak M(Id)Z\right|\frac{\varepsilon|X_\eta|}{|\beta_{0,\dots,0}|}+\sum_{\lambda\ne0}|\mathfrak M(\tau_\lambda)|\frac{\sqrt{|d|}^m}{Z^{\frac1{m-1}}}T_0|\tau_\lambda(X_\eta)|,
\end{equation}
where $Z=\sum_\mu A_\mu\omega^*_\mu$ as above.

The second term is a ratio of some constant to $Z^{\frac1{m-1}}$. Since $A_{0,\dots,0}=\Lambda_1$, the inequality
\eqref{el.finalapprox.z} shows that the threshold $N_0$ can be selected such that the bound
\begin{equation}\label{el.zused}
Z>\left(4\sum_{\lambda\ne0}|\mathfrak M(\tau_\lambda)\tau_\lambda(X_\eta)|\sqrt{|d|}^m T_0\right)^{m-1}
\end{equation}
holds, and then the second term in the right-hand side of \eqref{el.almostsys2} is less than $\frac14$.

Assume that such a threshold $N_0$ is selected. Calculate simultaneous approximations $A_\mu$, then compute $Z$.
Select $\varepsilon$ so that for each $\eta$ the inequality
\begin{equation}\label{el.epsilonused}
\varepsilon<\frac14\frac{|\beta_{0,\dots,0}|}{|\mathfrak M(Id)X_\eta|Z}.
\end{equation}
holds. Then the first term in the right-hand side of \eqref{el.almostsys2} is also less than $\frac14$.
Thus, the left-hand side of \eqref{el.almostsys2} is less than $\frac12$. Since $\sum_{\mu'}A_{\mu'}B_{\mu'\eta}\in\Z$,
we can recover the exact value of this sum by rounding $\mathfrak M(Id)Z\frac{\gamma X_\eta}{\beta_{0,\dots,0}}$ to an integer.

Now we obtain a system of linear equations for $b_\xi$ with the left-hand side
\begin{equation}\label{el.leftpartused}
\sum_\mu A_\mu B_{\mu\eta}=\sum_\xi\left(\sum_\mu A_\mu x_{\mu\xi\eta}\right)b_\xi.
\end{equation}

\begin{lemma} The matrix
$\left(\sum_\mu A_\mu x_{\mu\xi\eta}\right)_{\xi,\eta\in\{0,1\}^{t-1}}$
is nonsingular.
\end{lemma}
\begin{proof}
Assume that this matrix is singular. Equivalently, there exist numbers $y_\eta\in\Q$ such that
not all of them are zero and
\begin{equation}\label{el.goodmatrix1}
\sum_\eta\sum_\mu A_\mu x_{\mu\xi\eta}y_\eta=0.
\end{equation}

Fix some $\eta$. Consider the following square matrices:
\begin{eqnarray*}
(M_1)_{\mu_1\mu_2}&=&\tau_{\mu_1}(\omega_{\mu_2}),\\
(M_2)_{\mu'_1\mu'_2}&=&\tau_{\mu'_1}(\omega^*_{\mu'_2}),\\
(X)_{\mu''_1\mu''_2}&=&x_{\mu''_1\mu''_2\eta}
\end{eqnarray*}
and diagonal matrices $M_3$ with elements $\mathfrak M(\tau_\mu)$ and $M_4$
with elements $\tau_\mu(X_\eta)$. The equality \eqref{el.dual} can be interpreted
as matrix equality $M_1^TM_3M_2=E$, where $E$ is the identity matrix. In particular,
$M_1$, $M_2$ and $M_3$ are invertible. The set of all equalities obtained from
\eqref{el.xmultdef} under the action of all $\tau_\mu$, can be interpreted as
matrix equality $M_4M_1=M_1X$. Thus $X=M_1^{-1}M_4M_1$, $X^T=M_1^TM_4(M_1^T)^{-1}=
M_2^{-1}M_3^{-1}M_4M_3M_2$. Since any two diagonal matrices commute, $M_4M_3=M_3M_4$,
so $M_2X^T=M_4M_2$. Comparing the element in the line 1 and the column $\mu$, we
obtain
$$
\sum_\xi\omega_\xi^*x_{\mu\xi\eta}=X_\eta\omega^*_\mu.
$$

Now let $\eta$ vary. Multiply \eqref{el.goodmatrix1} by $\omega^*_\xi$ and
sum over all $\xi\in\{0,1\}^{t-1}$:
$$
\sum_\eta\sum_\mu A_\mu X_\eta\omega^*_\mu y_\eta=0,
$$
$$
\left(\sum_\eta X_\eta y_\eta\right)\left(\sum_\mu A_\mu\omega^*_\mu\right)=0.
$$
But the first factor is nonzero because $X_\eta$ are linearly independent over $\Q$
and not all of $y_\eta\in\Q$ are zero. The second factor is nonzero due to Theorem \ref{el.approxalg}.
The contradiction proves the lemma.
\end{proof}

So it is sufficient to solve a linear system $m\times m$ with nonsingular matrix to find $\{b_\mu\}$.
For example, one can use the standard Gaussian elimination.

Finally, we give an overall scheme for our optimization of the CM method.
\begin{enumerate}
\item Select numbers $q=p^n$, $\hat u,\hat v,D\in\Z$ as in the stage 1 of the basic algorithm
from Subsection \ref{basemethod}. The future curve will be defined over $\F_q$ and have
the order $q+1-\hat u$.
\item Enumerate all reduced forms. Calculate $T_0$ from \eqref{el.t0used}, $N_0$
from \eqref{el.zused}, using \eqref{el.finalapprox.z}. Apply the algorithm from
Section \ref{el.simapprox}.
\item Calculate the required precision $\varepsilon$ from \eqref{el.epsilonused}.
Calculate the polynomial $\hat H_D[j]$ by the definition \eqref{el.hdused}
approximately with the precision $\varepsilon$.
\item For each coefficient of the polynomial calculate the decomposition of
doubled real part as a $\Z$-linear combination of $\beta_\mu$.
In order to do this, obtain a system
of linear equations with the left-hand side \eqref{el.leftpartused} using
\eqref{el.almostsys2} and solve this system. Similarly calculate the decomposition
of doubled imaginary part as a $\Z$-linear combination of $\beta^*_\mu$.
(If the coefficient is known to be real, the stage for imaginary part is
not necessary and one can avoid doubling the real part.)
\item Reduce the polynomial modulo any prime ideal of $\ro_{K_G}$ lying above $p$,
obtain a polynomial over $\F_q$. Calculate any root in $\F_q$ (there always is one).
Construct an elliptic curve $E''$ over $\F_q$ with $j$-invariant equal to the found root.
\item If the order $E''$ is not the same as required, apply an isomorphism from
Subsection \ref{basemethod} (quadratic twist if $D<-4$).
\end{enumerate}
As in the original method, one can use another functions $\theta$ (described in Subsection \ref{el.otherinv})
instead of $j$. This requires correcting the bound $T_0$ as described in Section \ref{el.upbound},
using $\hat H_D[\theta,\alpha_*]$ instead of $\hat H_D[j]$, and calculating $j$-invariant
by the found value of $\theta$ as described in Subsection \ref{el.otherinv}.

\newcommand\BibAuthor[1]{#1}
\newcommand\BibTitle[1]{#1}
\newcommand\No{N}

\end{document}